\documentclass[11pt,reqno,a4paper]{amsart}

\usepackage[T1]{fontenc}
\usepackage[utf8]{inputenc}

\usepackage{libertinus}
\usepackage{microtype}

\usepackage{geometry}

\usepackage{xcolor}

\usepackage{amsmath,amssymb,amsthm,amsfonts}
\usepackage{mathtools}
\usepackage{mathrsfs}

\usepackage{graphicx}

\usepackage[
colorlinks=true,
linkcolor=blue,
citecolor=blue,
urlcolor=blue
]{hyperref}

\usepackage[nameinlink,noabbrev]{cleveref}

\usepackage{titlesec}

\titleformat{\section}
  {\normalfont\large\bfseries\centering}
  {\thesection.}
  {0.75em}
  {}

\titlespacing*{\section}
  {0pt}
  {3.2ex plus 1ex minus .2ex}
  {1.6ex plus .3ex}

\titleformat{\subsection}
  {\normalfont\normalsize\bfseries}
  {\thesubsection.}
  {0.75em}
  {}

\titlespacing*{\subsection}
  {0pt}
  {2.4ex plus .8ex minus .2ex}
  {1.0ex plus .2ex}

\titleformat{\subsubsection}
  {\normalfont\normalsize\bfseries\itshape}
  {\thesubsubsection.}
  {0.75em}
  {}

\titlespacing*{\subsubsection}
  {0pt}
  {2.0ex plus .6ex minus .2ex}
  {0.7ex plus .2ex}

\numberwithin{equation}{section}

\theoremstyle{plain}

\newtheorem{theorem}{Theorem}[section]
\newtheorem{lemma}[theorem]{Lemma}
\newtheorem{proposition}[theorem]{Proposition}
\newtheorem{corollary}[theorem]{Corollary}
\newtheorem*{conjecture}{Conjecture}

\newtheorem{maintheorem}{Theorem}

\theoremstyle{definition}

\theoremstyle{remark}

\newtheorem{remark}[theorem]{Remark}

\usepackage{etoolbox}

\AtBeginEnvironment{proof}{\vspace{6pt}}

\usepackage{enumitem}

\setlist{
itemsep=2pt,
topsep=4pt
}

\usepackage[toc,page]{appendix}

\newcommand{\SL}{\operatorname{SL}}

\newcommand{\SO}{\operatorname{SO}}

\newcommand{\GL}{\operatorname{GL}}

\newcommand{\rBC}{\operatorname{BC}} 

\newcommand{\cO}{\mathcal{O}}

\newcommand{\bC}{\mathbb{C}}

\newcommand{\bR}{\mathbb{R}}

\newcommand{\bZ}{\mathbb{Z}}

\newcommand{\goa}{\mathfrak{a}}

\def\acts{\curvearrowright}
\newcommand\subsetsim{\mathrel{%
\ooalign{\raise0.2ex\hbox{$\subset$}\cr\hidewidth\raise-0.8ex\hbox{\scalebox{0.9}{$\sim$}}\hidewidth\cr}}}

\newcommand{\rA}{\mathrm{A}}
\newcommand{\rB}{\mathrm{B}}
\newcommand{\rC}{\mathrm{C}}
\newcommand{\rD}{\mathrm{D}}
\newcommand{\rE}{\mathrm{E}}
\newcommand{\rF}{\mathrm{F}}
\newcommand{\rG}{\mathrm{G}}

\DeclareMathOperator{\Hom}{Hom}

\DeclareMathOperator{\Ind}{Ind}

\usepackage{comment}

\begin{document}

\title[Fourier Positivity for Spherical Functions I]{Fourier Positivity for Spherical Functions I: \\ Split Tori and spherical Principal Series}

\author[M. Bj\"orklund]{Michael Bj\"orklund}
\address{Department of Mathematical Sciences, 
Chalmers University of Technology and University of Gothenburg, 
SE-412 96 Gothenburg, Sweden}
\email{micbjo@chalmers.se}

\author[D. Liu]{Dongwen Liu}
\address{School of Mathematical Sciences,  Zhejiang University, 866 Yuhangtang Road, Hangzhou 310058, P.R. China}
\email{maliu@zju.edu.cn}
 
\author[J. Yu]{Jun Yu}
\address{School of Mathematical Sciences and Beijing International Center for Mathematical Research, Peking University, No. 5 Yiheyuan Road, Beijing 100871, P.R. China}
\email{junyu@bicmr.pku.edu.cn}

\author[G. Zhang]{Genkai Zhang}
\address{Department of Mathematical Sciences, 
Chalmers University of Technology and University of Gothenburg, 
SE-412 96 Gothenburg, Sweden}
\email{genkai@chalmers.se}


\begin{abstract}
We prove Fourier positivity for spherical functions on a semisimple linear algebraic group $G$ over a local field restricted to its split tori $A$
for unitary principal series parameters of $G$. For \(\SL_n(F)\), where \(F\) is a local field, we obtain an explicit recursive formula for the Fourier transform on the diagonal split torus in terms of local Rankin--Selberg factors for \(\GL_n\times\GL_{n-1}\), together with uniform exponential lower bounds in the spectral parameters.

The main input is a Plancherel expansion for the restriction of a \(\GL_n(F)\)-spherical function to \(\GL_{n-1}(F)\). Its coefficients are spherical periods computed by Rankin--Selberg theory. Positivity of the Fourier transform for general semisimple groups with unitary principal series parameters is obtained by reduction to full-rank subgroups of type \(\rA\). The results are motivated by variance non-vanishing problems for mixing abelian actions on homogeneous spaces.
\end{abstract}

\maketitle


\section{Introduction}

\subsection*{Restrictions to abelian subgroups}

Let \(G\) be a locally compact second countable topological group and let
\((\pi,V)\) be an irreducible unitary representation of \(G\). A fundamental
problem in representation theory is to understand the restriction of \(\pi\) to
a closed subgroup \(H<G\).

Suppose that \(V\) is cyclic, generated by a unit vector \(v_o\). The associated
matrix coefficient
\[
\psi_{v_o}(g)=\langle \pi(g)v_o,v_o\rangle
\]
is a continuous positive definite function on \(G\). Thus the restriction
problem is reflected in the restriction 
\(\psi_{v_o}|_H\)
 to \(H\).

When \(H\) is abelian, Bochner's theorem identifies \(\psi_{v_o}|_H\)
with the
Fourier transform of a non-negative measure on
the dual group \(\widehat H\). If
\(\psi_{v_o}|_H\in L^1(H)\), this measure is absolutely continuous, with
continuous non-negative density given by the Fourier transform of
\(\psi_{v_o}|_H\). It is therefore natural to ask when this Fourier transform is
strictly positive.

The paper concerns this question in the spherical setting. Thus, when \(G\) is
a semisimple group and \(K<G\) is a maximal compact subgroup, we shall focus on
bi-\(K\)-invariant positive definite functions. The precise quotient convention
needed for restrictions to abelian subgroups will be introduced below.

\subsection*{Motivation from homogeneous dynamics}

One source of the problem is the non-vanishing of variances in limit theorems
for abelian actions. Let \(H\) be an abelian lcsc group equipped with a proper
\(H\)-invariant metric, and suppose that Haar measure of metric balls \(B_R\)
grows polynomially in \(R\). Let \(H\acts (X,\mu)\) be an ergodic
probability-preserving action which is exponentially mixing of all orders, in
the sense of \cite{BEG20}. Then the arguments in \cite{BG20} imply that, for
suitable observables \(f\), the normalized averages
\[
\frac{1}{m_H(B_R)^{1/2}}
\left(
\int_{B_R} f(h.x)\,dm_H(h)
-
m_H(B_R)\int_X f\,d\mu
\right)
\]
converge in distribution to a normal law with variance
\[
\sigma_f^2
=
\int_H \psi_f(h)\,dh,
\qquad
\psi_f(h)=\int_X f(h.x)f(x)\,d\mu(x).
\]
The function \(\psi_f\) is positive definite on \(H\), and the mixing
assumptions imply that \(\psi_f\in L^1(H)\). Thus $\smash{\sigma_f^2=\widehat{\psi}_f(1)}$.
If \(\sigma_f^2=0\), the limiting distribution degenerates to the point mass at
\(0\). Thus, for such limit theorems to have non-trivial content, it is crucial
to establish positive variance.

In many Diophantine applications, one takes \(\smash{X=G/\Gamma}\), where \(G\)
is the group of \(F\)-points of a connected semisimple linear algebraic group
over a local field \(F\), and \(\Gamma<G\) is a lattice, while \(H\) is
contained in a split torus of \(G\). For local fields \(F\) of characteristic
zero, exponential mixing of all orders for such actions and smooth compactly
supported observables was established by the first author, Einsiedler and
Gorodnik in \cite{BEG20}. Proving positivity of the limiting variance is often
a separate problem, see e.g. \cite{BG19, BG23}. In lattice-counting applications, the observables are
typically special counting functions known as Siegel transforms, and positivity
is often proved using identities such as Siegel's mean value theorem and
Rogers' second moment formula
\cite{Rogers}.

The point of view in this paper gives a different route in the spherical case.
The function \(\psi_f\) is a matrix coefficient of the \(G\)-representation on
\(L^2(G/\Gamma)\). If \(K<G\) is a maximal compact subgroup and \(f\) is
\(K\)-invariant, then \(\psi_f\) is bi-\(K\)-invariant. Therefore, whenever
\(\psi_f|_H\in L^1(H)\), strict positivity of the Fourier transform of
\(\psi_f|_H\) at the trivial character implies
\[
\sigma_f^2=\widehat{\psi}_f(1)>0.
\]
This gives a representation-theoretic criterion for positive variance which is
not tied to special formulas for Siegel transforms, and applies to arbitrary
\(K\)-invariant observables whose matrix coefficients decay sufficiently fast.

There is also a related interpretation in terms of twisted cohomological
equations. Suppose, for instance, that \(G\) is a real Lie group and that
\(H=\{\exp(tX):t\in\bR\}\) is a non-compact one-parameter subgroup, with
\(X\in\mathfrak g\). If \(
\eta\in i\mathbb R\) is a unitary character of $H$, identified as
\(\eta(\exp(tX))=e^{\eta t}, t\in \mathbb R\), then vanishing of
the Fourier transform at \(\eta\) is closely connected to solvability of the
twisted equation
\[
v_o=(d\pi(X)-\eta)v.
\]
More precisely, assuming sufficient decay to justify integration by parts,
solvability of the twisted equation implies \(\smash{\widehat{\psi}_{v_o}\big|_H(\eta)=0}\), and conversely the
vanishing of \(\smash{\widehat{\psi}_{v_o}\big|_H(\eta)}\) implies solvability under the
moment condition appearing in \cite[Lemma~2.2]{ClB01}. Thus strict positivity
of these Fourier transforms gives a criterion excluding twisted coboundaries.
This connects the present question to the study of cohomological equations for
hyperbolic flows and higher-rank abelian actions; see, for example,
\cite{dlLMM86,KS94}.
\subsection*{The conjectural picture}

Let \(G\) be a semisimple linear group over a local field \(F\), let \(K<G\) be
a maximal compact subgroup, and let \(H<G\) be a closed abelian subgroup. If
\(\psi\) is bi-\(K\)-invariant, then its restriction to \(H\) is constant on
cosets of \(H\cap K\). Thus the natural abelian group on which to take the
Fourier transform is
\[
H_K:=H/(H\cap K).
\]
When \(\psi|_H\in L^1(H_K)\), we shall say that the pair \((H,\psi)\) has
\emph{Fourier positivity} if the Fourier transform of the corresponding
function on \(H_K\) is strictly positive on \(\widehat H_K\).

Motivated
by the above questions and after 
concrete computations
for various
cases, we arrived at the following working
conjecture.

\begin{conjecture}
Let \(H<G\) be a closed non-compact abelian
subgroup. Suppose that \(\psi\) is a continuous bi-\(K\)-invariant positive
definite function on \(G\) such that \(\psi|_H\in L^1(H_K)\). Then the pair
\((H,\psi)\) has Fourier positivity.
\end{conjecture}

The present paper concerns the first natural case 
of the conjecture: the
split
torus direction 
in $G$. If \(H = A<G\) is a maximal \(F\)-split torus
invariant under the Cartan involution, the relevant
abelian group is
\[
\mathsf A=A/(A\cap K).
\]
In the Archimedean case this quotient is naturally identified with \(A^\circ\),
while in the non-Archimedean case it is the usual cocharacter-lattice direction
of \(A\). We prove Fourier positivity for restrictions of spherical 
principal-series functions to this group \(\mathsf A\). Other parts of the
unitary spherical spectrum, including complementary series representations, will
be treated in subsequent papers in this series.

The non-compactness assumption is necessary. The quotient convention removes
the trivial example \(H<K\), since then \(H_K\) is the trivial group. But compact
subgroups not contained in \(K\) still give genuine counterexamples. For
example, let
\[
G=\SL_2(\bR),\qquad K=\SO(2),
\]
and let \(H=aKa^{-1}\) with \(a\notin N_G(K)\). Then \(H\) is compact abelian,
and for a generic element \(a\) the intersection \(H\cap K\) is finite. Hence
\[
H_K=H/(H\cap K)
\]
is again a circle. The constant function \(\psi\equiv 1\) is continuous,
positive definite, and bi-\(K\)-invariant, but its Fourier transform on \(H_K\)
is supported only at the trivial character. Thus strict positivity fails on
\(\widehat H_K\).

The assumption of full bi-\(K\)-invariance is also essential, and cannot be
weakened to \(K\)-finiteness. Indeed, let \(G\) be a real semisimple Lie group,
and let
\[
H=\{\exp(tX):t\in\bR\}
\]
be a non-compact one-parameter subgroup, with \(X\in\mathfrak g\). Let
\((\pi,V)\) be a unitary representation whose relevant \(K\)-finite matrix
coefficients decay sufficiently along \(H\). If \(w\) is a \(K\)-finite vector
and
\[
v_o=d\pi(X)w,
\]
then \(v_o\) is again \(K\)-finite, since the \(K\)-finite vectors form a
\((\mathfrak g,K)\)-module. The positive definite function
\[
\psi(g)=\langle \pi(g)v_o,v_o\rangle
\]
is therefore \(K\)-finite. If
\[
\varphi(t)=\langle \pi(\exp(tX))w,w\rangle,
\]
then \(\psi(\exp(tX))\) is, up to sign, the second derivative of \(\varphi(t)\).
Assuming sufficient decay, integration by parts gives
\[
\widehat{\psi}(1) = \int_H \psi(h)\,dh=0.
\]
Thus strict positivity may fail for \(K\)-finite positive definite functions
unless one assumes full bi-\(K\)-invariance.

The \(L^1\)-assumption is essential as well. For instance, let
\(G=\SL_2(\bC)\) and let \(N\simeq \bC\) be the upper unipotent subgroup. If
\(\phi_\lambda\) is a principal spherical function, then the restriction \(\phi_\lambda|_N\) is
positive definite on \(N\), but it is not integrable. Its Fourier transform can
nevertheless be defined away from the trivial character, and is given, up to a
positive scalar, by
\[
w\mapsto K_{i\lambda}(\|w\|)^2,
\]
where \(K_{i\lambda}\) denotes the \(K\)-Bessel function. Since \(K_{i\lambda}\)
has infinitely many positive real zeros, this transform vanishes on infinitely
many circles in \(\widehat N\simeq \bC\). The computation behind this example is
standard and will be discussed in a sequel devoted to unipotent restrictions,
including the role of the \smash{\(L^1\)}-assumption.

Finally, we note that the spherical Bochner theorem shows that if every non-trivial zonal
spherical function on \(G\) restricts to an \(L^1\)-function on \(H\), then the
conjecture for general bi-\(K\)-invariant positive definite functions reduces
to proving Fourier positivity for the restrictions of the non-trivial zonal
spherical functions themselves.

\subsection*{Notation and statement of results}

Let \(F\) be a local field and let \(G\) denote the group of \(F\)-points of a connected semisimple linear algebraic group defined over \(F\). Fix a maximal compact subgroup \(K<G\), a maximal \(F\)-split torus \(A<G\), and a positive Weyl chamber \(A^+\subset A\) such that
\[
G=K A^+ K .
\]
We also fix an Iwasawa decomposition
\[
G=KAN.
\]
Following the convention introduced above, we write
\[
\mathsf A=A/(A\cap K)
\]
for the split torus direction. In the Archimedean case this quotient is
naturally identified with \(A^\circ=\exp(\goa)\), where \(\goa\) is the Lie
algebra of \(A^\circ\). In the non-Archimedean case, there is no Lie algebra
involved; instead
\[
\goa=X_*(A)\otimes_\mathbb Z\bR,
\]
where \(X_*(A)\) is the cocharacter lattice of \(A\), and
\[
\goa^*=X^*(A)\otimes_\mathbb Z\bR.
\]
We denote by
\[
H_A:\mathsf A\to\goa
\]
the logarithm map in the Archimedean case and the valuation map in the
non-Archimedean case. Let \(W\) be the Weyl group, and let \(\rho\) denote the
half sum of positive roots, counted with multiplicities.

The relevant unitary dual of \(\mathsf A\) is
$\widehat{\mathsf A}_{\mathrm{ur}}
$ defined in
(\ref{hat-A}).
Thus, for \(\psi\in L^1(\mathsf A)\), we write
\[
\widehat{\psi}(\mu)
=
\int_{\mathsf A}
\psi(a)e^{-\mu(H_A(a))}\,da,
\qquad
\mu\in\widehat{\mathsf A}_{\mathrm{ur}}.
\]

For \(\lambda\in\goa_\bC^*\), let \(\phi_\lambda\) denote the Harish--Chandra
spherical function with parameter \(\lambda\). In this paper we consider the
positive definite spherical functions arising from the unitary spherical
principal series, corresponding to parameters 
$\smash{\lambda\in\widehat{\mathsf A}_{\mathrm{ur}}}$. For every such parameter, the restriction
$\phi_\lambda|_{\mathsf A}$
belongs to \(L^1(\mathsf A)\). \\

The first main result is the following positivity theorem.

\begin{maintheorem}
Let \(\lambda\in\widehat{\mathsf A}_{\mathrm{ur}}\), and let
\(\phi_\lambda\) be the corresponding positive definite spherical function on
\(G\). Then the Fourier transform of \(\phi_\lambda|_{\mathsf A}\) is strictly
positive on \(\widehat{\mathsf A}_{\mathrm{ur}}\).
\end{maintheorem}

By the spherical Bochner theorem, every continuous bi-\(K\)-invariant positive
definite function \(\psi\) on \(G\) admits a representation as a superposition
of positive definite spherical functions. We call the corresponding representing
measure the spherical Bochner measure of \(\psi\). Theorem~A has the following
consequence.

\begin{corollary}
Let \(\psi\) be a continuous bi-\(K\)-invariant positive definite function on
\(G\), and suppose that \(\smash{\psi|_{\mathsf A}\in L^1(\mathsf A)}\). Assume that the
spherical Bochner measure of \(\psi\) gives positive mass to the unitary
spherical principal series. Then the Fourier transform of \(\smash{\psi|_{\mathsf A}}\)
is strictly positive on \(\widehat{\mathsf A}_{\mathrm{ur}}\).
\end{corollary}

Indeed, after restricting to \(\mathsf A\), the spherical Bochner decomposition
is a positive superposition of positive definite functions on the abelian group
\(\mathsf A\). The part supported on the unitary spherical principal series
contributes strictly positively by Theorem~A, while the remaining part
contributes non-negatively.

The support condition in the corollary reflects the scope of the present paper,
rather than the expected limitation of the method. Other parts of the unitary
spherical spectrum, including complementary series representations, will be
treated in subsequent papers in this series.

In the body of the paper, especially in the \(\GL_n\)- and \(\SL_n\)-sections,
we often write \(A\), \(A_n\), or \(A_n'\) for these split directions
themselves. Thus compact diagonal factors in the non-Archimedean case, and the
compact sign or unitary factors in the Archimedean case, are suppressed in the
Fourier-transform notation.

\subsection*{Explicit lower bounds for \(\SL_n(F)\)}

We now describe the lower bounds more explicitly in the case
\[
G=\SL_n(F).
\]
Let \(A_n\) be the diagonal torus in \(\GL_n(F)\), and let
\[
A_n'=A_n\cap \SL_n(F)
\]
be the determinant-one diagonal torus. As above, the relevant split torus
direction is obtained by quotienting \(A_n'\) by its maximal compact subgroup;
in the Archimedean case this quotient is naturally identified with the identity
component. We suppress this distinction in the notation and write \(A_n'\) for
the split direction.

We shall use additive coordinates on the unitary dual of \(A_n'\). Put
\[
\frak I=
\begin{cases}
 i\bR, & F \text{ Archimedean},\\
 i\bR/(2\pi i/\log q_F)\bZ, & F \text{ non-Archimedean},
\end{cases}
\]
where \(q_F\) denotes the cardinality of the residue field of \(F\) in the
non-Archimedean case. Under this parametrization, the unitary dual of the split
direction \(A_n'\) is identified with $\frak I^n/\frak I$, 
where \(\frak I\) is embedded diagonally. We denote the quotient map by
\[
{\sf p}:\frak I^n\to \frak I^n/\frak I.
\]
If \(\lambda=(\lambda_1,\ldots,\lambda_n)\in\frak I^n\), we write
\begin{equation}
\label{lambda-bracket}
[\lambda]=\lambda_1+\cdots+\lambda_n.
\end{equation}
Thus every class in \(\frak I^n/\frak I\) has a representative satisfying $[\lambda]=0$.

For \(\bar\lambda\in\frak I^n/\frak I\), let \(\phi_{\bar\lambda}\) denote the
corresponding positive definite spherical function on \(\SL_n(F)\). We use the
local zeta function
\[
\zeta_F(s)=
\begin{cases}
\pi^{-s/2}\Gamma(s/2), & F=\bR,\\
2(2\pi)^{-s}\Gamma(s), & F=\bC,\\
(1-q_F^{-s})^{-1}, & F \text{ non-Archimedean}.
\end{cases}
\]
For \(\lambda\in\frak I^n\) and \(\nu\in\frak I^{n-1}\), define
\(
P_n(\lambda,\nu)\)
as in (\ref{P-n-lam-nu})
and  let \({\bf c}^{(n-1)}(\nu)\)
be the Harish-Chandra 
$c$-function given
in (\ref{HC-c}).
For parameters in the unitary range, the product
\[
P_n(\lambda,-\nu)\,\big|{\bf c}^{(n-1)}(\nu)\big|^{-2}
\]
is non-negative; after the usual limiting interpretation, it is strictly
positive away from the hyperplanes $\nu_k=\nu_l$.

We define \(c(\lambda,\mu)\) recursively on the fibers
\[
[\lambda]=[\mu].
\]
In rank \(1\), this condition forces \(\mu=\lambda\), and we set
\[
c(\lambda,\lambda)=1.
\]
For \(n\geq 2\), and
\begin{equation}
\label{lambda-up-}
\mu^-=(\mu_1,\ldots,\mu_{n-1}),
\end{equation}
set
\[
c(\lambda,\mu)
=
\int_{\substack{\nu\in\frak I^{n-1}\\ [\nu]=[\mu^-]}}
P_n(\lambda,-\nu)\,\big|{\bf c}^{(n-1)}(\nu)\big|^{-2}\,
c(\nu,\mu^-)\,d\nu .
\]
Here \(d\nu\) is the measure on the affine fiber obtained by translating a fixed
Haar measure on the corresponding linear fiber. See Section \ref{Sec3} and Section \ref{Sec4} for the normalization of Haar measures.

We fix once and for all norms on the parameter spaces \(\frak I^m\) and on the
quotients \(\frak I^m/\frak I\), using representatives in a fixed fundamental
domain in the non-Archimedean case. Since all such choices are equivalent in
fixed dimension, the precise norm is immaterial for the estimates below.

\begin{maintheorem}
Let \(G=\SL_n(F)\). For \(\bar\lambda,\bar\mu\in \frak I^n/\frak I\), choose
\[
\lambda\in{\sf p}^{-1}(\bar\lambda)
\qquad\text{with}\qquad
[\lambda]=0.
\]
Then
\[
\widehat{\phi}_{\bar\lambda}(\bar\mu)
=
\sum_{\substack{\mu\in{\sf p}^{-1}(\bar\mu)\\ [\mu]=0}}
c(\lambda,\mu).
\]
The sum is finite: it has one term in the Archimedean case and \(n\) terms in
the non-Archimedean case. In particular, there exist constants
\(C_n,D_n>0\) such that
\[
\widehat{\phi}_{\bar\lambda}(\bar\mu)
\geq
C_n e^{-D_n(\|\bar\lambda\|+\|\bar\mu\|)}
\]
for all \(\bar\lambda,\bar\mu\in\frak I^n/\frak I\).
\end{maintheorem}

The exponential lower bound is proved by applying the recursive formula
inductively, using lower bounds for the local zeta factors on explicit
positive-measure subsets of the integration domains. The recursive formula is
stated only in the \(\SL_n(F)\)-case, where the notation is transparent and the
lower bounds are completely explicit.

\subsection*{Some comments on the proof}

We briefly indicate the main steps in the proof in the case \(G=\SL_n(F)\). The
first step is to work on
\[
G_n=\GL_n(F)
\]
rather than directly on \(\SL_n(F)\). A unitary spherical parameter
\[
\bar\lambda\in\frak I^n/\frak I
\]
for \(\SL_n(F)\) is represented by a parameter
\[
\lambda\in\frak I^n
\]
for \(G_n\), and the corresponding spherical function on \(\SL_n(F)\) is the
restriction of the \(G_n\)-spherical function \(\Phi_\lambda^{(n)}\). Shifting
\(\lambda\) by a diagonal parameter only twists \(\Phi_\lambda^{(n)}\) by a
power of the determinant, and hence does not change its restriction to
\(\SL_n(F)\).

The key analytic input is a Plancherel expansion for the restriction of
\(\Phi_\lambda^{(n)}\) to
\[
G_{n-1}=\GL_{n-1}(F),
\]
embedded in \(G_n\) by
\[
g\longmapsto
\begin{pmatrix}
g&0\\
0&1
\end{pmatrix}.
\]
The coefficient of the spherical function \(\Phi_\nu^{(n-1)}\) in this expansion
is a spherical period for the pair \(G_n\times G_{n-1}\). This period is
computed by relating an open-orbit integral on a product of flag varieties to
the Rankin--Selberg integral for \(\GL_n\times\GL_{n-1}\). The result
is the kernel
\[
P_n(\lambda,-\nu)\big|{\bf c}^{(n-1)}(\nu)\big|^{-2}.
\]

There is a technical point in applying Plancherel. The restricted function
\[
g\longmapsto
\Phi_\lambda^{(n)}
\left(
\begin{pmatrix}
g&0\\
0&1
\end{pmatrix}
\right)
\]
does not lie in the Harish--Chandra Schwartz space of \(G_{n-1}\), so the usual
Plancherel formula cannot simply be invoked. Instead, we prove the required
expansion by pairing against compactly supported bi-\(K_{n-1}\)-invariant test
functions and using absolute convergence estimates for the period kernel. These
estimates are ultimately elementary consequences of Stirling's formula in the
Archimedean case and compactness of the unitary parameter space in the
non-Archimedean case.

The Plancherel expansion is then iterated along the chain
\[
A_n=\widetilde G_1\subset \widetilde G_2\subset\cdots\subset \widetilde G_n=G_n,
\qquad
\widetilde G_k=G_k\times G_1^{n-k}.
\]
At each step one keeps track of the central character. This produces the
recursive coefficient \(c(\lambda,\mu)\) and gives the Fourier expansion of
\(\Phi_\lambda^{(n)}\) on the diagonal torus of \(G_n\).

Finally, one passes back to \(\SL_n(F)\). Choosing a zero-sum representative
\([\lambda]=0\), and then pushing forward from the zero-sum hyperplane
\[
\{\mu\in\frak I^n:[\mu]=0\}
\]
to the quotient \(\frak I^n/\frak I\), gives the formula for
\(\widehat{\phi}_{\bar\lambda}(\bar\mu)\). In the Archimedean case this
pushforward has one point in each fiber; in the non-Archimedean case the fibers
have cardinality \(n\).

The exponential lower bound is proved from the same recursive formula. One
restricts the defining integral for \(c(\lambda,\mu)\) to explicit
positive-measure subsets of the affine fibers. These subsets are chosen by
translating a fixed set in the zero-sum fiber, so their measure is independent
of the parameter \(\mu\), and the factors involving
\(\nu_k-\nu_l\) remain uniformly controlled. The remaining local zeta factors
are bounded below using elementary Archimedean or non-Archimedean estimates.
Induction then gives the stated lower bound.

\section*{Acknowledgement} 

The authors thank Binyong Sun for addressing the proof of Lemma \ref{lem:Sun}, and thank Wee Teck Gan, Fan Gao and Xiaocheng Li for helpful discussions. M. Bj\"orklund was supported by the Swedish
research council
VR 11253322, and 
G. Zhang by VR 11253580. D. Liu was supported in part by National Key R \& D Program of China No. 2022YFA1005300. J. Yu was supported in part by the NSFC Grants 12431001 and 11971036.

\section{Preliminaries}

In this section we fix notation and recall the basic facts about local fields, split tori and spherical functions which will be used throughout the paper. We have tried to keep the discussion self-contained, but we shall freely use standard facts from harmonic analysis on reductive groups over local fields.

\subsection{Local fields and Fourier transforms}

Let \(F\) be a local field. We denote by \(|\cdot|_F\) its normalized absolute value. If \(F\) is non-Archimedean, we write \(\cO_F\) for its ring of integers, \(\mathfrak p_F\) for its maximal ideal, and
\[
q_F=\#(\cO_F/\mathfrak p_F)
\]
for the cardinality of its residue field.

We shall use the following local zeta function:
\begin{equation}
\label{zeta}
\zeta_F(s)=
\begin{cases}
\pi^{-s/2}\Gamma(s/2), & F=\bR,\\
2(2\pi)^{-s}\Gamma(s), & F=\bC,\\
(1-q_F^{-s})^{-1}, & F \text{ non-Archimedean}.
\end{cases}
\end{equation}
The local zeta functions appear naturally in the Harish--Chandra \(c\)-function and in the Rankin--Selberg factors used later.

Let \(V\) be a finite-dimensional real vector space and let \(V^*\) be its dual. We identify the unitary dual of \(V\), viewed as an additive group, with \(iV^*\). Thus \(\mu\in iV^*\) corresponds to the character
\[
x\mapsto e^{\mu(x)}.
\]
For \(f\in L^1(V)\), we use the Fourier transform convention
\[
\widehat f(\mu)
=
\int_V f(x)e^{-\mu(x)}\,dx,
\qquad \mu\in iV^*.
\]
The Haar measure \(dx\) will be fixed once and for all in each situation. With this convention, the Fourier transform of a positive definite \(L^1\)-function is non-negative.

\subsection{Cartan data and Fourier transforms on split tori}

Let \(G\) be the group of \(F\)-points of a connected semisimple linear algebraic group defined over \(F\). We fix a maximal compact subgroup \(K<G\), a maximal \(F\)-split torus \(A<G\), and a positive Weyl chamber \(A^+\subset A\) such that
\[
G=K A^+ K .
\]
We also fix a minimal parabolic subgroup \(P=MAN\), with split component \(A\), and write
\[
G=KAN
\]
for the corresponding Iwasawa decomposition.

We shall work with the non-compact split direction of \(A\). Thus we put
\[
\mathsf A=
\begin{cases}
A^\circ, & F \text{ Archimedean},\\
A/(A\cap K), & F \text{ non-Archimedean}.
\end{cases}
\]
In the Archimedean case, \(A^\circ=\exp(\goa)\), where \(\goa\) is the Lie algebra of \(A^\circ\). In the non-Archimedean case, there is no Lie algebra involved. Instead we set
\[
\goa=X_*(A)\otimes_\mathbb Z \mathbb R,
\]
where
\[
X_*(A)=\Hom_F(\mathbb G_m,A)
\]
is the cocharacter lattice of \(A\). Its dual is
\[
\goa^*=X^*(A)\otimes_\mathbb Z\mathbb R,
\]
where
\[
X^*(A)=\Hom_F(A,\mathbb G_m)
\]
is the character lattice. \\

The map
\[
H_A:\mathsf A\to\goa
\]
is the logarithm map in the Archimedean case. In the non-Archimedean case, it is the valuation map characterized by
\[
\langle \chi,H_A(a)\rangle
=
-\log_{q_F}|\chi(a)|_F,
\qquad
\chi\in X^*(A).
\]
Let \(\Sigma\subset \goa^*\) denote the restricted root system, let \(\Sigma^+\subset \Sigma\) be the set of positive roots determined by the positive Weyl chamber \(A^+\), and let \(W\) be the Weyl group. We write
\[
\rho=\frac12\sum_{\alpha\in\Sigma^+}m_\alpha\alpha,
\]
where \(m_\alpha\) denotes the multiplicity of the root \(\alpha\). \\

We identify the spherical unitary principal dual of \(\mathsf A\) with
\begin{equation}
\label{hat-A}
\widehat{\mathsf A}_{\mathrm{ur}}
=
\begin{cases}
i\goa^*, & F \text{ Archimedean},\\[2mm]
i\goa^*/(2\pi i/\log q_F)X^*(A), & F \text{ non-Archimedean}.
\end{cases}
\end{equation}
Thus, for \(f\in L^1(\mathsf A)\), we write
\[
\widehat f(\mu)
=
\int_{\mathsf A} f(a)e^{-\mu(H_A(a))}\,da,
\qquad \mu\in \widehat{\mathsf A}_{\mathrm{ur}}.
\]
In the non-Archimedean case the expression \(e^{-\mu(H_A(a))}\) is well-defined because \(\mu\) is taken modulo the lattice \((2\pi i/\log q_F)X^*(A)\).

\subsection{Spherical functions}

We recall the basic notation for spherical functions. Let \(\lambda\in \goa_\bC^*\). We denote by \(\phi_\lambda\) the Harish--Chandra spherical function with spectral parameter \(\lambda\), normalized by
\[
\phi_\lambda(e)=1.
\]
Thus \(\phi_\lambda\) is bi-\(K\)-invariant and satisfies
\[
D\phi_\lambda=\chi_\lambda(D)\phi_\lambda
\]
for every \(G\)-invariant differential operator \(D\) on \(G/K\) in the Archimedean case, with the analogous statement for the spherical Hecke algebra in the non-Archimedean case.

In this paper we shall mainly consider the positive definite spherical functions arising from the unitary spherical principal series. These correspond to parameters
\[
\lambda\in \widehat{\mathsf A}_{\mathrm{ur}}.
\]
For such \(\lambda\), the function \(\phi_\lambda\) is a continuous positive definite function on \(G\).

We shall frequently restrict spherical functions to the split torus direction \(\mathsf A\). For \(\lambda\in \widehat{\mathsf A}_{\mathrm{ur}}\), we write
\[
\widehat{\phi_\lambda}(\mu)
=
\int_{\mathsf A}
\phi_\lambda(a)e^{-\mu(H_A(a))}\,da,
\qquad
\mu\in \widehat{\mathsf A}_{\mathrm{ur}},
\]
whenever the integral is absolutely convergent. By the standard Harish--Chandra estimates, we see that \(\phi_\lambda|_{\mathsf A}\in L^1(\mathsf A)\) for every non-trivial unitary spherical parameter \(\lambda\); see \cite[Theorem 4.5.3]{W1} and \cite[Lemma II.1.1]{W03}. Hence $\smash{\widehat{\phi}_\lambda}$ is a continuous function on \(\smash{\widehat{\mathsf A}_{\mathrm{ur}}}\).

Recall the notion of weak containment introduced by J. Fell. A unitary representation $\pi$ of a Lie group $G$ is said to be weakly contained in a unitary representation $\sigma$ of $G$ if any diagonal matrix coefficient of $\pi$ can be approximated, uniformly on compact subsets of $G$, by convex combinations of diagonal matrix coefficients of $\sigma$. An irreducible unitary representation $\pi$ of $G$ is said to be tempered if it is weakly contained in the regular representation $L^{2}(G)$. 

\begin{proposition}\label{P:tempered-res}
Let $G$ be a Lie group, and let $H$ be a closed subgroup which is a type I group in the sense of von Neumann algebra. If $\pi$ is a tempered representation of $G$, then the support of $\pi|_{H}$
is contained in the tempered dual of $H$. 
\end{proposition}

\begin{proof}
Since $\pi$ is tempered, it is weakly contained in $L^{2}(G)$. Thus, 
$\pi|_{H}$ is weakly contained in $L^{2}(G)|_{H}$ as unitary representations of $H$; see \cite[Appendix F, Proposition F.3.4]{BHV08}. We also have that $L^{2}(G)|_{H}$ is weakly contained in $L^{2}(H)$ as unitary representations of $H$ by \cite[Proposition F.1.10]{BHV08}. Then according to \cite[Remark F.1.2(iv)]{BHV08}, $\pi|_{H}$ is also weakly contained in $L^{2}(H)$. That is to say, $\pi|_{H}$ is a tempered representation of $H$. 
\end{proof}

\begin{proposition}\label{P:tempered-sphe}
Let $G$ be a connected linear reductive Lie group. If an irreducible unitary representation $\pi$ of $G$ is both spherical and tempered, then it is isomorphic to the irreducible spherical  constituent of a spherical unitary principal series of $G$.  
\end{proposition}

\begin{proof}
As $\pi$ is tempered, there exists a cuspidal parabolic subgroup $P=MAN$, a discrete series $\sigma$ of $M$ and a unitary character $e^{\mu}$ of $A$ such that $\pi$ is an irreducible constituent of \[I_{P}(\sigma,\mu):=\Ind_{MAN}^{G}(\sigma\otimes e^{\mu}\otimes\mathbf{1}_{N});\] see 
\cite[Proposition 5.2.5]{W1}. Since $\pi$ is spherical, 
\[I_{P}(\sigma,\mu)|_{K}=\Ind_{K\cap M}^{K}(\sigma|_{K\cap M})\] contains the trivial representation of $K$. This forces that 
$\sigma_{K\cap M}$ contains the trivial representation of $K\cap M$. As $\sigma$ is a discrete series, it must hold that $M$ is compact and $\sigma$ is the trivial representation. Hence $P=MAN$ is a minimal parabolic subgroup and $I_{P}(\sigma,\mu)$ is a spherical unitary principal series. Moreover, $\pi$ is isomorphic to the unique irreducible spherical constituent of
$I_{P}(\sigma,\mu)$.  
\end{proof}

\section{The \(\GL_n\)-case} \label{Sec3}

In this section we begin the proof of the explicit recursive formula for \(\SL_n(F)\). It is convenient to work first with
\[
G_n=\GL_n(F),
\]
and then pass to
\[
G_n'=\SL_n(F).
\]
We write \(K_n<G_n\) for the standard maximal compact subgroup, \(A_n<G_n\) for the diagonal torus, and
\[
\overline B_n=A_n\overline N_n
\]
for the lower triangular Borel subgroup. As in the preliminaries, when Fourier transforms are taken on diagonal tori we suppress compact diagonal factors. Thus, in the non-Archimedean case, \(A_n\) denotes the diagonal split direction modulo \(A_n\cap K_n\), while in the Archimedean case it denotes the identity component when a Fourier transform is involved. For algebraic constructions, such as the definition of Borel subgroups and principal series, \(A_n\) denotes the full diagonal torus in \(G_n\). \\

Let
\[
\frak I=
\begin{cases}
 i\bR, & F \text{ Archimedean},\\
 i\bR/(2\pi i/\log q_F)\bZ, & F \text{ non-Archimedean}.
\end{cases}
\]

\subsection{Principal series and spherical vectors}

For \(\lambda=(\lambda_1,\ldots,\lambda_n)\in\frak I^n\), let \(\chi_\lambda\) be the spherical unitary character of \(A_n\) given by
\[
\chi_\lambda(\operatorname{diag}(a_1,\ldots,a_n))
=
\prod_{j=1}^n |a_j|_F^{\lambda_j}.
\]
We denote by
\[
\pi_\lambda=\Ind_{\overline B_n}^{G_n}(\chi_\lambda)
\]
the normalized spherical principal series representation
with $G_n$ acting
by the regular action from the right, $g.f(x)=f(xg)$.
Denote by $\pi_\lambda^\infty$ the subspace of smooth vectors.
Let \(f_\lambda\in\pi_\lambda\) be the normalized \(K_n\)-fixed vector such that $f_\lambda|_{K_n}\equiv 1$.

We pair \(\pi_\lambda\) with \(\pi_{-\lambda}\) by
\[
\langle f,f^\vee\rangle
=
\int_{K_n} f(k)f^\vee(k)\,dk,
\]
where \(dk\) denotes the Haar probability measure on \(K_n\). 
It can also be written as
\begin{equation}
\label{B-inte}
\langle f,f^\vee\rangle
=
\int_{
\mathcal B_n
} f(x)f^\vee(x) \Delta(x)dx
\end{equation}
where
\[
\mathcal B_n
=\overline B_n\backslash G_n
\]
is the flag variety of \(G_n\) and 
$\Delta(x)dx$ is
a $G_n$-quasi-invariant probability measure.
The corresponding normalized spherical function is
\[
\Phi_\lambda^{(n)}(g)
=
\langle g.f_\lambda,f_{-\lambda}\rangle,
\qquad g\in G_n.
\]

For \(t\in\frak I\), let \(\mathbf 1=(1,\ldots,1)\). Then
\[
\chi_{\lambda+t\mathbf 1}
=
\chi_\lambda|\det|_F^t,
\]
and hence
\[
\Phi_{\lambda+t\mathbf 1}^{(n)}(g)
=
|\det g|_F^t\Phi_\lambda^{(n)}(g),
\qquad g\in G_n.
\]
It follows that the restriction of \(\Phi_\lambda^{(n)}\) to \(G_n'\) depends only on the class of \(\lambda\) in
\[
\frak I^n/\frak I,
\]
where \(\frak I\) is embedded diagonally. We denote the resulting spherical function on \(G_n'\) by
\[
\phi_{\bar\lambda}^{(n)}
=
\Phi_\lambda^{(n)}\big|_{G_n'},
\qquad
\bar\lambda\in \frak I^n/\frak I.
\]

\subsection{Rankin--Selberg integrals}

Fix a non-trivial additive character \(\psi:F\to\bC^\times\) as in \cite[2.1.2]{Hu}. Let \(U_n<G_n\) denote the upper triangular unipotent subgroup, and define
\[
\psi_n:U_n\to\bC^\times,
\qquad
\psi_n(u)=\psi\left(\sum_{i=1}^{n-1}u_{i,i+1}\right).
\]
For \(f\in\pi_\lambda\), we define the Whittaker function by the Jacquet integral
\begin{equation}
\label{whitake-n}
W_f(g)
=
\int_{U_n} f(ug)\overline{\psi_n(u)}\,du,
\qquad g\in G_n,
\end{equation}
in the sense of holomorphic continuation; see \cite[Theorem 15.4.1]{W2}. Here \(du\) denotes the Haar measure on $U_n$ given by the product of the self-dual measure on $F$ with respect to $\psi$.

For the normalized spherical vector \(f_\lambda\in\pi_\lambda\), we have
\begin{equation}\label{eq:whittaker-normalization}
W_{f_\lambda}|_{K_n}
\equiv
\frac{1}
{\prod_{1\leq i<j\leq n}\zeta_F(1+\lambda_j-\lambda_i)}.
\end{equation}
This spherical Whittaker function is right \(K_n\)-invariant and has the displayed constant value on \(K_n\); see \cite{Hu}. \\

Assume now that \(n\geq 2\), and let
\(
\nu=(\nu_1,\ldots,\nu_{n-1})\in\frak I^{n-1}.
\)
For \(f'\in\pi_\nu\), we write
\begin{equation}
    \label{whitake-n-1}
W_{f'}(g)
=
\int_{U_{n-1}} f'(ug)\psi_{n-1}(u)\,du,
\qquad g\in G_{n-1},
\end{equation}
again in the sense of holomorphic continuation.
Here for the \(G_{n-1}\)-representation we use the opposite Whittaker character
instead of
(\ref{whitake-n})
for $G_n$, as is customary in the Rankin--Selberg integral below. 
If  \(f=f_\nu\in\pi_\nu\)
is spherical we use the analogous normalization in rank \(n-1\). \\

The Rankin--Selberg integral is
\[
Z(s,W_f,W_{f'})
=
\int_{U_{n-1}\backslash G_{n-1}}
W_f\left(
\begin{bmatrix}
g&0\\
0&1
\end{bmatrix}
\right)
W_{f'}(g)
|\det g|_F^{s-\frac12}\,dg,
\]
where the quotient measure on $U_{n-1}\backslash G_{n-1}$ is fixed as in \cite[2.1.4]{Hu}.
By \cite{JPSS,J}, this integral converges absolutely for \(\operatorname{Re}(s)\) sufficiently large and admits meromorphic continuation to \(s\in\bC\).

We shall use that the spherical Whittaker functions are test vectors for the Rankin--Selberg integral. More precisely, by \cite{JPSS81,St,Hu},
\begin{equation}\label{eq:RS-spherical}
Z(s,W_{f_\lambda},W_{f_\nu})
=
W_{f_\lambda}(1_n)W_{f_\nu}(1_{n-1})
L(s,\pi_\lambda\times\pi_\nu),
\end{equation}
where
\[
L(s,\pi_\lambda\times\pi_\nu)
=
\prod_{i=1}^{n}\prod_{k=1}^{n-1}
\zeta_F(s+\lambda_i+\nu_k).
\]
We shall also use the following adjoint \(L\)-factor:
\[
L(s,\pi_\lambda,\mathrm{Ad})
=
\prod_{i,j=1}^{n}
\zeta_F(s+\lambda_i-\lambda_j).
\]
With this convention the diagonal factors contribute powers of \(\zeta_F(s)\), which are independent of \(\lambda\) and will always be absorbed into constants depending only on \(n\).

\subsection{Spherical periods}
Consider 
diagonal right 
action 
of \(G_{n-1}\) 
on \[ 
\mathcal B_n\times \mathcal B_{n-1}.
\]
By \cite[Lemma 1.1]{LLSS} it 
has a unique open orbit. We choose representatives
\[
(\overline B_n z_n,\overline B_{n-1}z_{n-1})
\]
for this orbit, where the matrices \(z_m\in \GL_m(\bZ)\) are defined recursively by
\(z_1=[1]=1,
\)
and, for \(m\geq 2\),
\[
z_m
=
\begin{bmatrix}
w_{m-1}&0\\
0&1
\end{bmatrix}
\begin{bmatrix}
z_{m-2}^{-1}&0\\
0&1_2
\end{bmatrix}
\begin{bmatrix}
z_{m-1}^{\rm t}w_{m-1}z_{m-1}&e_{m-1}^{\rm t}\\
0&1
\end{bmatrix}.
\]
Here \(w_m\) denotes the \(m\times m\) anti-diagonal permutation matrix, \(e_m=[0\ \cdots\ 0\ 1]\), and \({}^{\rm t}\) denotes transpose. The stabilizer of
\[
(\overline B_n z_n,\overline B_{n-1}z_{n-1})
\]
in \(G_{n-1}\) is trivial.

The following lemma was communicated to us by Binyong Sun.

\begin{lemma}\label{lem:Sun}
Let \(\lambda\in\frak I^n\) and \(\nu\in\frak I^{n-1}\). There is a constant
\(\kappa_n>0\), independent of \(\lambda\) and \(\nu\), such that
\[
\begin{aligned}
\langle f,f^\vee\rangle \langle f',f'^\vee\rangle
&=
\kappa_n
\int_{G_{n-1}}
f\left(
z_n
\begin{bmatrix}
h&0\\
0&1
\end{bmatrix}
\right)
f^\vee\left(
z_n
\begin{bmatrix}
h&0\\
0&1
\end{bmatrix}
\right) \\
&\hspace{35mm}\cdot
f'(z_{n-1}h)f'^\vee(z_{n-1}h)\,dh
\end{aligned}
\]
for all pairs
\(
(f, 
f^\vee)\in
\pi_\lambda^\infty
\times 
\pi_{-\lambda}^\infty
\)
and \(
(f',
f'^\vee)\in
\pi_\nu^\infty
\times
\pi_{-\nu}^\infty
\)
in the principal series of 
\(G_{n}\) and
\(G_{n-1}\) respectively.
\end{lemma}

\begin{proof}
We use the induced models for the two groups. Thus \(f\) and \(f^\vee\) are
smooth functions on \(G_n\), with prescribed left equivariance under
\(\overline B_n\), while \(f'\) and \(f'^\vee\) are smooth functions on
\(G_{n-1}\), with prescribed left equivariance under
\(\overline B_{n-1}\).

Since \(f\in\pi_\lambda\) and \(f^\vee\in\pi_{-\lambda}\), the inducing
characters cancel in the product \(ff^\vee\).
More precisely, because we use
normalized induction, the product \(ff^\vee\) transforms under \(\overline B_n\)
by the density character for the quotient
\(
\mathcal B_n=\overline B_n\backslash G_n.
\)
Hence \(ff^\vee\) defines a smooth $L^1$-density on \(\mathcal B_n\). Similarly,
since \(f'\in\pi_\nu\) and \(f'^\vee\in\pi_{-\nu}\), now for \(G_{n-1}\), the
product \(f'f'^\vee\) defines a smooth density on
\[
\mathcal B_{n-1}=\overline B_{n-1}\backslash G_{n-1}.
\]
Here a smooth density means a smooth section of the density bundle; in local
analytic coordinates it is of the form
\[
\delta(x)|dx_1\cdots dx_d|,
\]
with \(\delta\) smooth, and in the non-Archimedean case locally constant; the
$L^1$-integrability
is true for all $(f, f^\vee)\in 
(\pi_{\lambda}\times
\pi_{\lambda})
$ by 
(\ref{B-inte}) and
the Cauchy-Schwartz 
inequality.

Thus
\[
(ff^\vee)\otimes(f'f'^\vee)
\]
defines a smooth
$L^1$-density on
\[
X=\mathcal B_n\times\mathcal B_{n-1}.
\]
With our fixed normalizations of measures, the product of the standard pairings
is a positive constant multiple of the integral of this density over \(X\):
\[
\langle f,f^\vee\rangle\langle f',f'^\vee\rangle
=
C_0
\int_X (ff^\vee)\otimes(f'f'^\vee),
\]
where \(C_0>0\) depends only on the normalizations of the measures on the flag
varieties and on \(K_n,K_{n-1}\). In particular, \(C_0\) is independent of
\(\lambda\) and \(\nu\).

The varieties \(\mathcal B_n\) and \(\mathcal B_{n-1}\) are flag varieties,
hence smooth and irreducible. Therefore
\[
X=\mathcal B_n\times\mathcal B_{n-1}
\]
is smooth and irreducible. By \cite[Lemma 1.1]{LLSS}, the diagonal action of
\(G_{n-1}\) on \(X\) has a unique open orbit, represented by
\[
(\overline B_n z_n,\overline B_{n-1}z_{n-1}),
\]
and the stabilizer of this point is trivial. Let \(X^\circ\subset X\) denote
this open orbit. Since \(X^\circ\) is nonempty and Zariski open in the
irreducible variety \(X\), it is Zariski dense. Hence
\[
X\setminus X^\circ
\]
is a proper Zariski-closed subset of positive codimension.

We use the standard fact that the \(F\)-points of a positive-codimension
Zariski-closed subset of a smooth \(F\)-variety have measure zero with respect
to every smooth density. Therefore
\[
\int_X (ff^\vee)\otimes(f'f'^\vee)
=
\int_{X^\circ} (ff^\vee)\otimes(f'f'^\vee).
\]
Since the stabilizer of
\[
(\overline B_n z_n,\overline B_{n-1}z_{n-1})
\]
is trivial, the map
\[
j:G_{n-1}\longrightarrow X^\circ,
\qquad
j(h)=
\left(
\overline B_n z_n
\begin{bmatrix}
h&0\\
0&1
\end{bmatrix},
\overline B_{n-1}z_{n-1}h
\right)
\]
identifies \(G_{n-1}\) with the open orbit.

Pulling the density \((ff^\vee)\otimes(f'f'^\vee)\) back by \(j\), and comparing
it with the fixed Haar measure \(dh\) on \(G_{n-1}\), gives
\[
j^*\big((ff^\vee)\otimes(f'f'^\vee)\big)
=
C_1\, Q(h) dh, 
\, Q(h):=
f\left(
z_n
\begin{bmatrix}
h&0\\
0&1
\end{bmatrix}
\right)
f^\vee\left(
z_n
\begin{bmatrix}
h&0\\
0&1
\end{bmatrix}
\right)
f'(z_{n-1}h)f'^\vee(z_{n-1}h),
\]
where \(C_1>0\) depends only on the normalization of densities and on the Haar
measure \(dh\). Thus
\[
\int_X (ff^\vee)\otimes(f'f'^\vee)
=
C_1
\int_{G_{n-1}}
Q(h)\,dh .
\]

Combining the preceding identities and setting
\[
\kappa_n=C_0C_1,
\]
we obtain
\[
\langle f,f^\vee\rangle \langle f',f'^\vee\rangle
=
\kappa_n
\int_{G_{n-1}}
Q(h)\,dh .
\]
The constant \(\kappa_n\) is positive and is independent of \(\lambda\) and
\(\nu\), because after passing to densities on the flag varieties the inducing
characters have cancelled and only the fixed measure normalizations remain.
Taking \(f^\vee=\overline f\) and \(f'^\vee=\overline{f'}\) also makes the
positivity of \(\kappa_n\) evident.
\end{proof}

For \(f\in\pi_\lambda\) and \(f'\in\pi_\nu\), define the open-orbit integral
\begin{equation}
\label{Lambda}
\Lambda(f,f')
=
\int_{G_{n-1}}
f\left(
z_n
\begin{bmatrix}
g&0\\
0&1
\end{bmatrix}
\right)
f'(z_{n-1}g)\,dg .
\end{equation}
By \cite[Proposition 1.4]{LLSS}, this integral converges absolutely.

We write
\[
\iota_{n-1}:G_{n-1}\longrightarrow G_n,
\qquad
\iota_{n-1}(g)=
\begin{pmatrix}
g&0\\
0&1
\end{pmatrix}.
\]
For \(\lambda\in\frak I^n\) and \(\nu\in\frak I^{n-1}\), we consider the spherical period
\begin{equation} \label{period}
\mathcal P_n(\lambda,\nu)
=
\int_{G_{n-1}}
\Phi_\lambda^{(n)}(\iota_{n-1}(g))
\Phi_\nu^{(n-1)}(g)\,dg,
\end{equation}
which converges absolutely by \cite[Proposition 2.1]{H}.

The following theorem computes this period. In the non-Archimedean case, this is given in \cite[Theorem 2.12]{H} as the unramified calculation of the refined Gan-Gross-Prasad conjecture 
(see also \cite{GGP}) for unitary groups at split places. Our proof is based on \cite{LLSS} and is uniform for all local fields, and the proof also gives the absolute convergence of the defining integral
\eqref{period}.
 
\begin{theorem}
\label{thm:bessel-period}
For \(\lambda\in\frak I^n\) and \(\nu\in\frak I^{n-1}\), the integral defining
\(\mathcal P_n(\lambda,\nu)\) converges absolutely. Moreover, there is a constant \(\kappa'_n>0\), independent of \(\lambda\) and \(\nu\), such that
\[
\mathcal P_n(\lambda,\nu)
=
\kappa'_n
\frac{
L\left(\frac12,\pi_\lambda\times\pi_\nu\right)
L\left(\frac12,\pi_{-\lambda}\times\pi_{-\nu}\right)
}{
L(1,\pi_\lambda,\mathrm{Ad})
L(1,\pi_\nu,\mathrm{Ad})
}.
\]
Equivalently, after absorbing the diagonal factors into the constant,
\[
\mathcal P_n(\lambda,\nu)
=
\kappa'_n
\frac{
\prod_{i=1}^{n}\prod_{k=1}^{n-1}
\zeta_F(\frac12+\lambda_i+\nu_k)
\zeta_F(\frac12-\lambda_i-\nu_k)
}{
\prod_{i\neq j}
\zeta_F(1+\lambda_i-\lambda_j)
\prod_{k\neq l}
\zeta_F(1+\nu_k-\nu_l)
}.
\]
\end{theorem}

\begin{proof}
Let
\[
(f, f^\vee )
\in \pi_\lambda
\pi_{-\lambda},
\qquad
(f', f'^\vee )
\in \pi_\nu \times
\pi_{-\nu}.
\]
Lemma~\ref{lem:Sun},
the absolute
convergence of
\ref{Lambda},
 Fubini's theorem and
a change of variables
give
\[
\int_{G_{n-1}}
\left\langle
\begin{bmatrix}
g&0\\
0&1
\end{bmatrix}. f,
f^\vee
\right\rangle
\langle g. f',f'^\vee\rangle\,dg
=
\kappa_n\Lambda(f,f')\Lambda(f^\vee,f'^\vee).
\]
By \cite[Theorem 1.6]{LLSS}, with the present normalization of spherical vectors and Whittaker functions, we have
\[
\Lambda(f,f')
=
\prod_{i+k\leq n}
\frac{
\zeta_F(\frac12-\lambda_i-\nu_k)
}{
\zeta_F(\frac12+\lambda_i+\nu_k)
}
\,
Z\left(\frac12,W_f,W_{f'}\right),
\]
and
\[
\Lambda(f^\vee,f'^\vee)
=
\prod_{i+k\leq n}
\frac{
\zeta_F(\frac12+\lambda_i+\nu_k)
}{
\zeta_F(\frac12-\lambda_i-\nu_k)
}
\,
Z\left(\frac12,W_{f^\vee},W_{f'^\vee}\right).
\]
By \cite[Proposition 1.4]{LLSS}, the integrals \(\Lambda(f,f')\) and
\(\Lambda(f^\vee,f'^\vee)\) converge absolutely. Applying Lemma~\ref{lem:Sun}
to compactly supported truncations of the open orbit and then passing to the
limit gives
\[
\int_{G_{n-1}}
\left\langle
\begin{bmatrix}
g&0\\
0&1
\end{bmatrix}. f,
f^\vee
\right\rangle
\langle g. f',f'^\vee\rangle\,dg
=
\kappa_n\Lambda(f,f')\Lambda(f^\vee,f'^\vee),
\]
and the integral on the left is absolutely convergent. Here \(\kappa_n>0\)
depends only on the normalizations.

With the present normalization, any scalar discrepancy between the intertwining
normalizations used in \cite{LLSS} and the Whittaker normalization
\eqref{eq:whittaker-normalization} is independent of \(\lambda\) and \(\nu\);
we absorb it into \(\kappa_n\).

Multiplying these identities, the extra zeta ratios cancel, and hence
\[
\int_{G_{n-1}}
\left\langle
\begin{bmatrix}
g&0\\
0&1
\end{bmatrix}. f,
f^\vee
\right\rangle
\langle g. f',f'^\vee\rangle\,dg
=
\kappa_n
Z\left(\frac12,W_f,W_{f'}\right)
Z\left(\frac12,W_{f^\vee},W_{f'^\vee}\right).
\]
We now specialize to spherical vectors:
\[
f=f_\lambda,\qquad
f^\vee=f_{-\lambda},\qquad
f'=f_\nu,\qquad
f'^\vee=f_{-\nu}.
\]
Then the left-hand side is precisely \(\mathcal P_n(\lambda,\nu)\). By the Rankin--Selberg spherical test-vector formula stated in \eqref{eq:RS-spherical},
\[
Z\left(\frac12,W_{f_\lambda},W_{f_\nu}\right)
=
W_{f_\lambda}(1_n)W_{f_\nu}(1_{n-1})
L\left(\frac12,\pi_\lambda\times\pi_\nu\right),
\]
and similarly
\[
Z\left(\frac12,W_{f_{-\lambda}},W_{f_{-\nu}}\right)
=
W_{f_{-\lambda}}(1_n)W_{f_{-\nu}}(1_{n-1})
L\left(\frac12,\pi_{-\lambda}\times\pi_{-\nu}\right).
\]
By \eqref{eq:whittaker-normalization}, evaluated at \(1_n\), we have
\[
W_{f_\lambda}(1_n)
=
\frac{1}
{\prod_{1\leq i<j\leq n}\zeta_F(1+\lambda_j-\lambda_i)}.
\]
Applying the same formula with \(-\lambda\) in place of \(\lambda\), we get
\[
W_{f_{-\lambda}}(1_n)
=
\frac{1}
{\prod_{1\leq i<j\leq n}\zeta_F(1-\lambda_j+\lambda_i)}.
\]
The first product contains the factors with differences \(\lambda_j-\lambda_i\)
for \(i<j\), while the second contains the opposite differences. Hence their
product is
\[
W_{f_\lambda}(1_n)W_{f_{-\lambda}}(1_n)
=
\frac{1}
{\prod_{i\neq j}\zeta_F(1+\lambda_i-\lambda_j)}.
\]
Equivalently, if
\[
L(1,\pi_\lambda,\mathrm{Ad})
=
\prod_{i,j=1}^{n}\zeta_F(1+\lambda_i-\lambda_j),
\]
then
\[
W_{f_\lambda}(1_n)W_{f_{-\lambda}}(1_n)
=
\frac{\zeta_F(1)^n}
{L(1,\pi_\lambda,\mathrm{Ad})}.
\]
The factor \(\zeta_F(1)^n\) is independent of \(\lambda\), and is absorbed into the constant. The same argument, with \(n\) replaced by \(n-1\), gives
\[
W_{f_\nu}(1_{n-1})W_{f_{-\nu}}(1_{n-1})
=
\frac{\zeta_F(1)^{n-1}}
{L(1,\pi_\nu,\mathrm{Ad})}.
\]
Therefore the product of the four spherical Whittaker values contributes
\[
\frac{\kappa_n''}
{
L(1,\pi_\lambda,\mathrm{Ad})
L(1,\pi_\nu,\mathrm{Ad})
},
\]
where \(\kappa_n''>0\) depends only on \(n\) and on the normalizations. Therefore
\[
\mathcal P_n(\lambda,\nu)
=
\kappa'_n
\frac{
L\left(\frac12,\pi_\lambda\times\pi_\nu\right)
L\left(\frac12,\pi_{-\lambda}\times\pi_{-\nu}\right)
}{
L(1,\pi_\lambda,\mathrm{Ad})
L(1,\pi_\nu,\mathrm{Ad})
}
\]
for some \(\kappa'_n>0\) independent of \(\lambda\) and \(\nu\).

Finally, expanding the \(L\)-factors gives
\[
L\left(s,\pi_\lambda\times\pi_\nu\right)
=
\prod_{i=1}^{n}\prod_{k=1}^{n-1}
\zeta_F(s+\lambda_i+\nu_k),
\]
and
\[
L(s,\pi_\lambda,\mathrm{Ad})
=
\prod_{i,j=1}^{n}
\zeta_F(s+\lambda_i-\lambda_j).
\]
Substituting these expressions gives the displayed product formula.
\end{proof}

\begin{remark}\label{rem:Pn}
For unitary parameters, the formula in Theorem~\ref{thm:bessel-period}
can be written in terms of absolute values. Indeed, for our local zeta
functions one has
\[
\overline{\zeta_F(s)}=\zeta_F(\overline{s}).
\]
Hence, if \(\lambda_i,\nu_k\in\frak I\), then
\[
\zeta_F\left(\frac12+\lambda_i+\nu_k\right)
\zeta_F\left(\frac12-\lambda_i-\nu_k\right)
=
\left|\zeta_F\left(\frac12+\lambda_i+\nu_k\right)\right|^2.
\]
Likewise, the factors in the adjoint \(L\)-functions pair as
\[
\zeta_F(1+\lambda_i-\lambda_j)
\zeta_F(1+\lambda_j-\lambda_i)
=
\left|\zeta_F(1+\lambda_i-\lambda_j)\right|^2,
\]
and similarly for the \(\nu\)-variables. The diagonal terms \(i=j\) and
\(k=l\) contribute powers of \(\zeta_F(1)\), which are independent of
\(\lambda\) and \(\nu\), and are absorbed into the positive constant.

Thus, after absorbing positive constants depending only on \(n\) and on the
normalization of Haar measures, the kernel attached to the spherical period is
\begin{equation}
\label{P-n-lam-nu}
P_n(\lambda,\nu)
=
\frac{
\prod_{i=1}^{n}\prod_{k=1}^{n-1}
\left|\zeta_F\!\left(\frac12+\lambda_i+\nu_k\right)\right|^2
}{
\prod_{1\leq i<j\leq n}
\left|\zeta_F(1+\lambda_i-\lambda_j)\right|^2
\prod_{1\leq k<l\leq n-1}
\left|\zeta_F(1+\nu_k-\nu_l)\right|^2
}.
\end{equation}
In the recursive formulas below, all positive constants relating
\(\mathcal P_n\) and \(P_n\) are absorbed into the Haar measures on the parameter
spaces.
\end{remark}

Thus, throughout the rest of this section, we write
\[
P_n(\lambda,\nu)
=
\frac{
\prod_{i=1}^{n}\prod_{k=1}^{n-1}
\left|\zeta_F\!\left(\frac12+\lambda_i+\nu_k\right)\right|^2
}{
\prod_{1\leq i<j\leq n}
\left|\zeta_F(1+\lambda_i-\lambda_j)\right|^2
\prod_{1\leq k<l\leq n-1}
\left|\zeta_F(1+\nu_k-\nu_l)\right|^2
}.
\]


\subsection{The restricted spherical function}

We shall use the spherical Plancherel formula on \(G_{n-1}\). With our normalizations, if
\[
u\in C_c^\infty(K_{n-1}\backslash G_{n-1}/K_{n-1}),
\]
then
\[
u(g)
=
\int_{\frak I^{n-1}/S_{n-1}}
\widehat u(\nu)\,
\Phi_\nu^{(n-1)}(g)\,
\big|{\bf c}^{(n-1)}(\nu)\big|^{-2}\,d\nu,
\]
where
\[
\widehat u(\nu)
=
\int_{G_{n-1}}
u(g)\Phi_{-\nu}^{(n-1)}(g)\,dg,
\]
and
\begin{equation}
\label{HC-c}
{\bf c}^{(n-1)}(\nu)
=
\prod_{1\leq k<l\leq n-1}
\frac{\zeta_F(\nu_k-\nu_l)}
{\zeta_F(1+\nu_k-\nu_l)}
\end{equation}
is the Harish-Chandra
$c$-function.
As usual, changing Haar measures only changes the Plancherel measure by a positive constant depending on \(n\), and such constants will be absorbed below.

We now restrict \(\Phi_\lambda^{(n)}\) to \(G_{n-1}\), embedded in \(G_n\) by
\[
\iota_{n-1}(g)
=
\begin{pmatrix}
g&0\\
0&1
\end{pmatrix}.
\]
Thus we consider the bi-\(K_{n-1}\)-invariant function
\begin{equation}
\label{f-lam}    
F_\lambda(g)
=
\Phi_\lambda^{(n)}(\iota_{n-1}(g)),
\qquad g\in G_{n-1}.
\end{equation}
The function \(F_\lambda\) does not in general belong to the Harish--Chandra Schwartz space of \(G_{n-1}\). Hence the usual Plancherel formula cannot be applied directly. The following proposition gives the corresponding Plancherel expansion. Its coefficients are the spherical periods computed in Theorem~\ref{thm:bessel-period} above.

Before stating the proposition, we now fix the norm convention used in the estimates from this point onward. For 
\[
\xi=(\xi_1,\ldots,\xi_m)\in\frak I^m, 
\]
we write
\[
\|\xi\|=\sum_{j=1}^m |\xi_j|.
\]
In the non-Archimedean case, where
\[
\frak I=i\bR/(2\pi i/\log q_F)\bZ,
\]
we choose representatives in a fixed fundamental domain and use the corresponding quotient distance. The same convention is used for quotient spaces such as \(\frak I^m/\frak I\). Since the non-Archimedean parameter spaces are compact, the precise choices are immaterial for the estimates.

All Haar measures and finite Weyl-group factors are fixed once and for all. Changing these choices only changes the formulas below by positive constants depending on \(n\), and such constants will be absorbed into the normalizations used later.

\begin{proposition}\label{prop:restricted-plancherel}
Let \(\lambda\in\frak I^n\). Then 
the function
$F_\lambda$ in (\ref{f-lam}) is
 given by
\begin{equation}
    \label{inversion}
F_\lambda(g)
=
\int_{\frak I^{n-1}/S_{n-1}}
\Phi_\nu^{(n-1)}(g)
\,
\mathcal P_n(\lambda,-\nu)\,
\big|{\bf c}^{(n-1)}(\nu)\big|^{-2}\,
d\nu,\,
g\in G_{n-1}.
\end{equation}
The integral converges absolutely and locally uniformly in \(g\).
\end{proposition}

The absolute convergence follows from the following estimate, together with
Remark~\ref{rem:Pn}.

\begin{lemma}\label{lem:period-plancherel-estimate}
For every \(0<\varepsilon\leq 1\), there exist \(C_{n,\varepsilon}>0\) and $\eta_n > 0$ such that
\[
P_n(\lambda,-\nu)
\big|{\bf c}^{(n-1)}(\nu)\big|^{-2}
\leq
C_{n,\varepsilon}
\exp\left(
\varepsilon\|\lambda\|-\eta_n\varepsilon\|\nu\|
\right)
\]
for all \(\lambda\in\frak I^n\) and \(\nu\in\frak I^{n-1}\).
\end{lemma}

The proof uses Stirling's formula together with two elementary inequalities for real numbers. We record the inequalities first.

\begin{lemma}
\label{lem:Yu}
Let \(x_1,\ldots,x_n\) and \(y_1,\ldots,y_{n-1}\) be real numbers. Then
\[
\sum_{1\leq k<l\leq n-1}|y_k-y_l|
\leq
\left(1-\frac{2}{n}\right)
\sum_{i=1}^{n}\sum_{k=1}^{n-1}|x_i-y_k|,
\]
and
\[
\sum_{1\leq i<j\leq n}|x_i-x_j|
+
\sum_{1\leq k<l\leq n-1}|y_k-y_l|
\leq
\sum_{i=1}^{n}\sum_{k=1}^{n-1}|x_i-y_k|.
\]
\end{lemma}

\begin{proof} We prove
the first inequality
by induction.
For \(n=2\), the assertion is trivial. Assume \(n\geq 3\).
For every pair \(k<l\) and every \(i\), the triangle inequality gives
\[
|y_k-y_l|
\leq
|y_k-x_i|+|x_i-y_l|.
\]
Summing over \(i=1,\ldots,n\), we get
\[
n|y_k-y_l|
\leq
\sum_{i=1}^n |x_i-y_k|
+
\sum_{i=1}^n |x_i-y_l|.
\]
Now sum this inequality over all pairs \(1\leq k<l\leq n-1\). Each fixed \(y_k\) occurs in exactly \(n-2\) pairs, and hence
\[
n\sum_{1\leq k<l\leq n-1}|y_k-y_l|
\leq
(n-2)\sum_{i=1}^n\sum_{k=1}^{n-1}|x_i-y_k|.
\]
Dividing by \(n\) proves the claim.

To prove the 
second inequality 
we order the \(2n-1\) numbers
\[
x_1,\ldots,x_n,y_1,\ldots,y_{n-1}
\]
as
\[
z_1\leq z_2\leq \cdots \leq z_{2n-1}.
\]
If some of the numbers are equal, the corresponding intervals have length zero, so the ordering among equal elements is irrelevant. Each summand in both sides
of the inequality 
is a sum of 
$z_{r+1}-z_r$
and we shall
rewrite the sums
of $z_{r+1}-z_r$, $r=1, \cdots, 2n-2$.

For each $r$ let
\[
a_r=\#\{i:x_i\leq z_r\},
\qquad
b_r=\#\{k:y_k\leq z_r\}.
\]
We compare the coefficients of the interval length \(z_{r+1}-z_r\) in the two sides of the desired inequality.

Any summand
in the \[\text{RHS}= 
\sum_{i=1}^{n}\sum_{k=1}^{n-1}|x_i-y_k|
\]
contributes 
to the term \(z_{r+1}-z_r\) to its distance exactly when its two endpoints lie on different sides of the interval \((z_r,z_{r+1})\). Hence the coefficient of \(z_{r+1}-z_r\) in
RHS
is
\[
a_r(n-1-b_r)+b_r(n-a_r).
\]
Similarly, the coefficient of \(z_{r+1}-z_r\) in
\[
\sum_{1\leq i<j\leq n}|x_i-x_j|
+
\sum_{1\leq k<l\leq n-1}|y_k-y_l|
\]
is
\[
a_r(n-a_r)+b_r(n-1-b_r).
\]
The difference between the first coefficient and the second is
\[
a_r(n-1-b_r)+b_r(n-a_r)
-
a_r(n-a_r)-b_r(n-1-b_r).
\]
After simplification this equals
\[
(a_r-b_r)(a_r-b_r-1).
\]
Since \(a_r-b_r\) is an integer, this quantity is always non-negative. Therefore the coefficient of every interval length \(z_{r+1}-z_r\) is at least as large on the right-hand side as on the left-hand side. Summing over \(r\) proves the inequality.
\end{proof}

\begin{proof}[Proof of Lemma~\ref{lem:period-plancherel-estimate}]
By the definitions of \(P_n\) and \({\bf c}^{(n-1)}\), we have
\[
P_n(\lambda,-\nu)|{\bf c}^{(n-1)}(\nu)|^{-2}
=
\frac{
\prod_{i=1}^{n}\prod_{k=1}^{n-1}
\left|\zeta_F\!\left(\frac12+\lambda_i-\nu_k\right)\right|^2
}{
\prod_{1\leq i<j\leq n}
\left|\zeta_F(1+\lambda_i-\lambda_j)\right|^2
\prod_{1\leq k<l\leq n-1}
\left|\zeta_F(\nu_k-\nu_l)\right|^2
}.
\]
Here the Plancherel density cancels the factors
\[
\zeta_F(1+\nu_k-\nu_l)
\]
which occur in \(P_n\).

If \(F\) is non-Archimedean, then the parameter space
\[
\frak I^n\times\frak I^{n-1}
\]
is compact. The possible singular hyperplanes in the displayed expression come from local zeta factors. The factors which occur in the denominator are never zero on the unitary axis; they either remain finite and nonzero, or they have poles, in which case their reciprocals vanish. Thus, after assigning the natural limiting value along these hyperplanes, the expression is bounded on the compact parameter space. Hence the desired estimate follows after increasing \(C_{n,\varepsilon}\).

Assume from now on that \(F\) is Archimedean. In 
this case we identify \(i\mathbb R\) with \(\mathbb R\) by taking imaginary parts when writing absolute values. Put
\[
L=\sum_{1\leq i<j\leq n}|\lambda_i-\lambda_j|,
\qquad
N=\sum_{1\leq k<l\leq n-1}|\nu_k-\nu_l|,
\qquad
M=\sum_{i=1}^{n}\sum_{k=1}^{n-1}|\lambda_i-\nu_k|.
\]
By Stirling's formula,
\[
|\Gamma(x+it)|
\asymp_x
(1+|t|)^{x-\frac12}e^{-\frac{\pi}{2}|t|}
\]
for fixed \(x\). Since
\[
\zeta_\bR(s)=\pi^{-s/2}\Gamma(s/2),
\qquad
\zeta_\bC(s)=2(2\pi)^{-s}\Gamma(s),
\]
it follows that the expression above is bounded by a polynomial factor in
\(\lambda,\nu\), times
\[
\exp\left(
\alpha(L+N-M)
\right),
\qquad
\alpha=\frac{\pi}{[\bC:F]}.
\]
Polynomial factors may be absorbed into an arbitrarily small exponential loss. Thus, for every \(\delta>0\), after increasing the constant it is enough to estimate
\[
\alpha(L+N-M)+\delta(L+N+M).
\]
We now use the two elementary inequalities. Lemma~\ref{lem:Yu} gives
\[
L+N-M\leq 0,
\]
and
\[
N\leq \left(1-\frac2n\right)M,
\]
and hence
\[
N-M\leq -\frac2n M.
\]
Choose a number \(0<\theta<1\). We split
\[
\alpha(L+N-M)
=
(1-\theta)\alpha(L+N-M)
+
\theta\alpha(L+N-M).
\]
The first term is non-positive by Lemma~\ref{lem:Yu}. For the second term we use
\[
L+N-M=L+(N-M)\leq L-\frac2n M,
\]
and obtain
\[
\theta\alpha(L+N-M)
\leq
\theta\alpha L-\frac{2\theta\alpha}{n}M.
\]
Moreover, since \(N\leq M\), we have
\[
\delta(L+N+M)\leq \delta L+2\delta M.
\]
Taking
\[
\delta=\frac{\theta\alpha}{2n},
\]
we get
\[
\alpha(L+N-M)+\delta(L+N+M)
\leq
\left(\theta\alpha+\delta\right)L
-
\frac{\theta\alpha}{n}M.
\]
Finally,
\[
L\leq (n-1)\|\lambda\|
\]
and
\[
M
=
\sum_{i=1}^{n}\sum_{k=1}^{n-1}|\lambda_i-\nu_k|
\geq
n\|\nu\|-(n-1)\|\lambda\|.
\]
Therefore there are constants \(B_n,E_n>0\), depending only on \(n\) and \(F\), such that
\[
\alpha(L+N-M)+\delta(L+N+M)
\leq
B_n\theta\|\lambda\|-E_n\theta\|\nu\|.
\]
Choosing \(\theta>0\) so that \(B_n\theta\leq\varepsilon\), and increasing the constant once more, we obtain
\[
P_n(\lambda,-\nu)|{\bf c}^{(n-1)}(\nu)|^{-2}
\leq
C_{n,\varepsilon}
\exp\left(
\varepsilon\|\lambda\|-\eta_n\varepsilon\|\nu\|
\right)
\]
for some \(\eta_n>0\) depending only on \(n\) and \(F\). This proves the estimate.
\end{proof}

\begin{proof}[Proof of Proposition~\ref{prop:restricted-plancherel}]
Recall 
(\ref{f-lam}).
We first prove that the right-hand side of the claimed formula is well-defined.
By Theorem~\ref{thm:bessel-period} and Remark~\ref{rem:Pn}, there is a constant
\(d_n>0\), depending only on \(n\) and on the Haar normalizations, such that
\[
\mathcal P_n(\lambda,-\nu)=d_nP_n(\lambda,-\nu)
\]
on the unitary axis. Hence Lemma~\ref{lem:period-plancherel-estimate} gives
\[
\mathcal P_n(\lambda,-\nu)
\big|{\bf c}^{(n-1)}(\nu)\big|^{-2}
\leq
C_{n,\varepsilon}
\exp\left(
\varepsilon\|\lambda\|-\eta_n\varepsilon\|\nu\|
\right),
\]
for some $\eta_n > 0$. 
Since \(\lambda\) is fixed in the proposition, the factor
\(e^{\varepsilon\|\lambda\|}\) may be absorbed into the constant. \\

Since \(\Phi_\nu^{(n-1)}\) is a positive definite spherical function,
\[
|\Phi_\nu^{(n-1)}(g)|\leq 
|\Phi_\nu^{(n-1)}(1)|=1
\]
for every \(g\in G_{n-1}\). Therefore
\[
\left|
\Phi_\nu^{(n-1)}(g)
\mathcal P_n(\lambda,-\nu)
\big|{\bf c}^{(n-1)}(\nu)\big|^{-2}
\right|
\leq
C_{\lambda,n,\varepsilon}
\exp\left(
-\eta_n\varepsilon\|\nu\|
\right).
\]
Thus the right-hand side 
of (\ref{inversion})
is integrable on \(\frak I^{n-1}/S_{n-1}\), and the majorant is independent of \(g\),
namely 
\[
\widetilde F_\lambda(g)
:=
\int_{\frak I^{n-1}/S_{n-1}}
\Phi_\nu^{(n-1)}(g)
\mathcal P_n(\lambda,-\nu)
\big|{\bf c}^{(n-1)}(\nu)\big|^{-2}\,d\nu
\]
defines a continuous bi-\(K_{n-1}\)-invariant function on \(G_{n-1}\). \\

We prove that \(\widetilde F_\lambda=F_\lambda\). Let
\[
u\in C_c^\infty(K_{n-1}\backslash G_{n-1}/K_{n-1}).
\]
By absolute convergence and Fubini's theorem,
\[
\begin{aligned}
\int_{G_{n-1}}
\widetilde F_\lambda(g)
u(g)
\,dg
&=
\int_{\frak I^{n-1}/S_{n-1}}
\left(
\int_{G_{n-1}}
\Phi_\nu^{(n-1)}(g) u(g)
\,dg
\right)
\mathcal P_n(\lambda,-\nu)
\big|{\bf c}^{(n-1)}(\nu)\big|^{-2}\,d\nu .
\end{aligned}
\]
By the definition of the spherical transform,
\[
\int_{G_{n-1}}
\Phi_\nu^{(n-1)}(g)u(g)\,dg
=
\widehat u(-\nu).
\]
Therefore
\[
\int_{G_{n-1}}\widetilde F_\lambda(g)u(g)\,dg
=
\int_{\frak I^{n-1}/S_{n-1}}
\widehat u(-\nu)
\mathcal P_n(\lambda,-\nu)
\big|{\bf c}^{(n-1)}(\nu)\big|^{-2}\,d\nu .
\]
By the definition of the spherical period,
\[
\mathcal P_n(\lambda,-\nu)
=
\int_{G_{n-1}}
F_\lambda(h)\Phi_{-\nu}^{(n-1)}(h)\,dh .
\]
Hence
\[
\begin{aligned}
\int_{G_{n-1}}\widetilde F_\lambda(g)u(g)\,dg
&=
\int_{\frak I^{n-1}/S_{n-1}}
\left(
\int_{G_{n-1}}
F_\lambda(h)\Phi_{-\nu}^{(n-1)}(h)\,dh
\right)
\widehat u(-\nu)
\big|{\bf c}^{(n-1)}(\nu)\big|^{-2}\,d\nu .
\end{aligned}
\]
Using 
Theorem 
\ref{thm:bessel-period}
the rapid decay of \(\widehat u\)
in the Archimedean case 
\cite[Theorem 6.6.8]{GV}
and compactness of the parameter space in the non-Archimedean case, we obtain
by Fubini's Theorem that
\[
\begin{aligned}
\int_{G_{n-1}}\widetilde F_\lambda(g)u(g)\,dg
&=
\int_{G_{n-1}}
F_\lambda(h)
\left(
\int_{\frak I^{n-1}/S_{n-1}}
\widehat u(-\nu)
\Phi_{-\nu}^{(n-1)}(h)
\big|{\bf c}^{(n-1)}(\nu)\big|^{-2}\,d\nu
\right)
dh .
\end{aligned}
\]
Changing variables \(\eta=-\nu\) in the inner integral, and using the invariance
of the Plancherel measure under \(\nu\mapsto-\nu\), the inner integral becomes
\[
\int_{\frak I^{n-1}/S_{n-1}}
\widehat u(\eta)
\Phi_{\eta}^{(n-1)}(h)
\big|{\bf c}^{(n-1)}(\eta)\big|^{-2}\,d\eta.
\]
By the spherical Plancherel formula (\cite{HC3, GV}), this is \(u(h)\).
 Hence
\[
\int_{G_{n-1}}\widetilde F_\lambda(g)u(g)\,dg
=
\int_{G_{n-1}}F_\lambda(g)u(g)\,dg .
\]
Thus \(F_\lambda\) and \(\widetilde F_\lambda\) have the same pairing against every compactly supported bi-\(K_{n-1}\)-invariant smooth function \(u\). Their difference is a continuous bi-\(K_{n-1}\)-invariant function. If it were nonzero at some point, then by continuity and bi-\(K_{n-1}\)-invariance it would have nonzero pairing with a compactly supported bi-\(K_{n-1}\)-invariant test function supported in a small neighborhood of the corresponding double coset, a contradiction. Hence the difference is identically zero.
\end{proof}


\section{Fourier coefficients} \label{Sec4}

In this section we pass from the restricted Plancherel formula to the Fourier expansion on the diagonal torus. The point of the notation below is to keep track of the central character at each step of the recursion.

In view of Remark~\ref{rem:Pn}, we absorb the positive constants relating \(\mathcal P_n\) and \(P_n\) into the Haar measures on the parameter spaces. Thus all recursive formulas below are written in terms of \(P_n\). We also absorb the finite Weyl-group factors into the normalization of Haar measure on the parameter spaces. Thus, from now on, we write the recursive integrals over the corresponding fibers in \(\frak I^m\), rather than over quotients by symmetric groups.

Here \(A_n\) denotes the full diagonal torus of \(G_n=\GL_n(F)\). The spherical characters \(\chi_\mu\) used below are insensitive to the maximal compact subgroup of \(A_n\): to the sign components in the real case, to the unitary part in the complex case, and to \(A_n\cap K_n\) in the non-Archimedean case. Thus, when taking Fourier transforms, \(A_n\) is understood through this split direction.

We define a chain of subgroups of \(G_n\) by
\[
A_n=\widetilde G_1\subset \widetilde G_2\subset\cdots\subset \widetilde G_n=G_n,
\]
where
\[
\widetilde G_k=G_k\times G_1^{n-k}
\]
is viewed as a subgroup of \(G_n\) by block diagonal embedding. Thus
\(\widetilde G_1=G_1^n=A_n\). \\

Recall the notation
\(
[\lambda]
\) and 
\(
\lambda^{-}
\) defined in 
(\ref{lambda-bracket})
and 
(\ref{lambda-up-})
for  \(\lambda\in\frak I^n\).  
Let \(\nu\in\frak I^{n-1}\)
and define a function on
\[
\widetilde G_{n-1}=G_{n-1}\times G_1
\]
by
\begin{equation}\label{Philn}
\Phi_{\lambda,\nu}(g,a)
=
\Phi_\nu^{(n-1)}(g)\,|a|_F^{[\lambda]-[\nu]},
\qquad
g\in G_{n-1},\ a\in G_1.
\end{equation}
The scalar factor is exactly what adjusts the central character. Indeed, the function \(\Phi_\nu^{(n-1)}\) has central character \(|\cdot|_F^{[\nu]}\) on \(G_{n-1}\), while \(\Phi_\lambda^{(n)}\) has central character \(|\cdot|_F^{[\lambda]}\) on \(G_n\).

The following is the form of Proposition~\ref{prop:restricted-plancherel} which will be iterated. With the normalization convention fixed above, we write the Plancherel integral over \(\frak I^{n-1}\) rather than over \(\frak I^{n-1}/S_{n-1}\).

\begin{corollary}\label{cor:inversion}
Assume \(n\geq 2\). For \(\lambda\in\frak I^n\), one has
\[
\Phi_\lambda^{(n)}(\tilde g)
=
\int_{\frak I^{n-1}}
\Phi_{\lambda,\nu}(\tilde g)\,
P_n(\lambda,-\nu)\,
\big|{\bf c}^{(n-1)}(\nu)\big|^{-2}\,d\nu,
\qquad
\tilde g\in\widetilde G_{n-1}.
\]
The integral converges absolutely and locally uniformly in \(\tilde g\).
\end{corollary}

\begin{proof}
Let
\[
\tilde g=(g,a)\in G_{n-1}\times G_1,
\]
viewed as the block diagonal matrix
\[
\tilde g=
\begin{pmatrix}
g&0\\
0&a
\end{pmatrix}
\in G_n.
\]
Writing \(a\) also for the scalar matrix \(aI_{n-1}\) in \(G_{n-1}\), we have
\[
\tilde g
=
aI_n
\begin{pmatrix}
a^{-1}g&0\\
0&1
\end{pmatrix}.
\]
Since \(\Phi_\lambda^{(n)}\) has central character \(|\cdot|_F^{[\lambda]}\), Proposition~\ref{prop:restricted-plancherel} gives
\[
\begin{aligned}
\Phi_\lambda^{(n)}(\tilde g)
&=
|a|_F^{[\lambda]}
\Phi_\lambda^{(n)}
\left(
\begin{pmatrix}
a^{-1}g&0\\
0&1
\end{pmatrix}
\right)\\
&=
|a|_F^{[\lambda]}
\int_{\frak I^{n-1}}
\Phi_\nu^{(n-1)}(a^{-1}g)\,
P_n(\lambda,-\nu)\,
\big|{\bf c}^{(n-1)}(\nu)\big|^{-2}\,d\nu .
\end{aligned}
\]
Here we have used the measure normalization described above, so that \(P_n\) appears instead of \(\mathcal P_n\). Since \(\Phi_\nu^{(n-1)}\) has central character \(|\cdot|_F^{[\nu]}\), we have
\[
\Phi_\nu^{(n-1)}(a^{-1}g)
=
|a|_F^{-[\nu]}\Phi_\nu^{(n-1)}(g).
\]
Therefore
\[
\Phi_\lambda^{(n)}(\tilde g)
=
\int_{\frak I^{n-1}}
\Phi_\nu^{(n-1)}(g)P_n(\lambda,-\nu)\,
|a|_F^{[\lambda]-[\nu]}
\big|{\bf c}^{(n-1)}(\nu)\big|^{-2}\,d\nu .
\]
By the definition of \(\Phi_{\lambda,\nu}\), this is precisely
\[
\Phi_\lambda^{(n)}(\tilde g)
=
\int_{\frak I^{n-1}}
\Phi_{\lambda,\nu}(\tilde g)\,
P_n(\lambda,-\nu)\,
\big|{\bf c}^{(n-1)}(\nu)\big|^{-2}\,d\nu .
\]
The same estimate as in Proposition~\ref{prop:restricted-plancherel} gives
absolute and locally uniform convergence, since the extra central factor
\[
|a|_F^{[\lambda]-[\nu]}
\]
has absolute value \(1\).
\end{proof}

We now define the recursive Fourier coefficient. For \(n=1\), the condition \([\lambda]=[\mu]\) forces \(\mu=\lambda\), and we set
\[
c_1(\lambda,\lambda)=1.
\]
For \(n\geq 2\), and for \(\lambda,\mu\in\frak I^n\) satisfying
\[
[\lambda]=[\mu],
\]
define
\[
c_n(\lambda,\mu)
=
\int_{\substack{\nu\in\frak I^{n-1}\\ [\nu]=[\mu^-]}}
P_n(\lambda,-\nu)\,
\big|{\bf c}^{(n-1)}(\nu)\big|^{-2}
c_{n-1}(\nu,\mu^-)\,d\nu.
\]
For simplicity
we write $c(\lambda, 
\nu)=c(\lambda, \nu)$
and the dependence
on $n$ will be clear
in the context. 

Here \(d\nu\) denotes the Haar measure on the affine fiber
\[
\{\nu\in\frak I^{n-1}:[\nu]=[\mu^-]\}.
\]

The following estimates justify the recursive definition and the repeated use of Fubini's theorem.

\begin{lemma}\label{lem:est}
For every \(n\geq 1\), the integral defining \(c(\lambda,\mu)\) converges absolutely for all
\(\lambda,\mu\in\frak I^n\) with \([\lambda]=[\mu]\). Moreover, the following
estimates hold:

\smallskip
\noindent\textbf{Pointwise bound.}
For every \(0<\varepsilon\leq 1\), there exists \(C'_{n,\varepsilon}>0\) such that
\begin{equation}
    \label{4.2-1}
c(\lambda,\mu)
\leq
C'_{n,\varepsilon}e^{\varepsilon\|\lambda\|}
\end{equation}[
for all \(\lambda,\mu\in\frak I^n\) with \([\lambda]=[\mu]\). In addition,
\(c(\lambda,\mu)\) is continuous in \(\lambda\) and \(\mu\).

\smallskip
\noindent\textbf{Integrability in the second variable.}
For fixed \(\lambda\), the function \(c(\lambda,\cdot)\) is integrable on the affine hyperplane
\[
\{\mu\in\frak I^n:[\mu]=[\lambda]\}.
\]
More precisely, for every \(0<\varepsilon\leq 1\), there exists \(C''_{n,\varepsilon}>0\) such that
\begin{equation}
    \label{4.2-2}
\int_{\substack{\mu\in\frak I^n\\ [\mu]=[\lambda]}}
c(\lambda,\mu)\,d\mu
\leq
C''_{n,\varepsilon}e^{\varepsilon\|\lambda\|}.
\end{equation}
\end{lemma}

\begin{proof}
The proof is by induction on \(n\). For \(n=1\), the assertions are immediate. Assume \(n\geq 2\), and assume the result known in rank \(n-1\). We use Lemma~\ref{lem:period-plancherel-estimate} in the form
\[
P_n(\lambda,-\nu)\big|{\bf c}^{(n-1)}(\nu)\big|^{-2}
\leq
C_{n,\delta}
\exp\left(
\delta\|\lambda\|-\eta_n\delta\|\nu\|
\right),
\]
valid for every \(0<\delta\leq 1\), where \(\eta_n>0\) depends only on \(n\). \\

Choose \(0<\delta\leq 1\) so small that
\[
\delta\leq \varepsilon
\qquad\text{and}\qquad
\frac12\eta_n\delta\leq 1 .
\]
By the induction hypothesis, applied with \(\frac12\eta_n\delta\), we have
\[
c(\nu,\mu^-)
\leq
C'_{n-1,\frac12\eta_n\delta}
\exp\left(
\frac12\eta_n\delta\|\nu\|
\right).
\]
Therefore the integrand defining \(c(\lambda,\mu)\) is bounded by
\[
C_{n,\delta}
C'_{n-1,\frac12\eta_n\delta}
\exp\left(
\delta\|\lambda\|
-\frac12\eta_n\delta\|\nu\|
\right).
\]
This is integrable over the affine fiber
\[
\{\nu\in\frak I^{n-1}:[\nu]=[\mu^-]\}.
\]
Since \(\delta\leq\varepsilon\), this gives
\[
c(\lambda,\mu)
\leq
C'_{n,\varepsilon}e^{\varepsilon\|\lambda\|}.
\]
Continuity follows from dominated convergence. To apply it, identify nearby fibers
\[
\{\nu:[\nu]=[\mu^-]\}
\]
by translation in the diagonal direction; the estimates above give a locally uniform integrable majorant, and the induction hypothesis gives continuity of \(c(\nu,\mu^-)\). \\

It remains to prove the estimate
(\ref{4.2-2}). By Tonelli's theorem,
\begin{equation}
\label{integ-c}
    \begin{aligned}
&\int_{\substack{\mu\in\frak I^n\\ [\mu]=[\lambda]}}
c(\lambda,\mu)\,d\mu  \\
&=
\int_{\substack{\mu\in\frak I^n\\ [\mu]=[\lambda]}}
\int_{\substack{\nu\in\frak I^{n-1}\\ [\nu]=[\mu^-]}}
P_n(\lambda,-\nu)
\big|{\bf c}^{(n-1)}(\nu)\big|^{-2}
c(\nu,\mu^-)\,d\nu\,d\mu .
\end{aligned}
\end{equation}
For fixed \(\nu\), the condition
\[
[\nu]=[\mu^-]
\]
means that \(\mu^-\) ranges over the affine hyperplane
\[
\{\xi\in\frak I^{n-1}:[\xi]=[\nu]\}.
\]
Once \(\mu^-=\xi\) is fixed, the last coordinate \(\mu_n\) is determined by the condition
\[
[\mu]=[\lambda].
\]
Thus changing
the order of integration 
(\ref{integ-c})
becomes
\[
\int_{\nu\in\frak I^{n-1}}
P_n(\lambda,-\nu)
\big|{\bf c}^{(n-1)}(\nu)\big|^{-2}
\left(
\int_{\substack{\xi\in\frak I^{n-1}\\ [\xi]=[\nu]}}
c(\nu,\xi)\,d\xi
\right)d\nu .
\]
By the induction hypothesis, again with parameter \(\frac12\eta_n\delta\),
\[
\int_{\substack{\xi\in\frak I^{n-1}\\ [\xi]=[\nu]}}
c(\nu,\xi)\,d\xi
\leq
C''_{n-1,\frac12\eta_n\delta}
\exp\left(
\frac12\eta_n\delta\|\nu\|
\right).
\]
Combining this with Lemma~\ref{lem:period-plancherel-estimate}, we obtain
\[
\int_{\substack{\mu\in\frak I^n\\ [\mu]=[\lambda]}}
c(\lambda,\mu)\,d\mu
\leq
C''_{n,\varepsilon}e^{\delta\|\lambda\|}
\leq
C''_{n,\varepsilon}e^{\varepsilon\|\lambda\|}.
\]
This completes the proof.
\end{proof}

For \(\mu=(\mu_1,\ldots,\mu_n)\in\frak I^n\), we write
\[
\chi_\mu(\operatorname{diag}(a_1,\ldots,a_n))
=
\prod_{j=1}^n |a_j|_F^{\mu_j},
\qquad
\operatorname{diag}(a_1,\ldots,a_n)\in A_n .
\]
These characters are trivial on the maximal compact subgroup of \(A_n\), so the formula depends only on the diagonal split direction. \\

Applying Corollary~\ref{cor:inversion} recursively along the chain
\[
A_n=\widetilde G_1\subset\widetilde G_2\subset\cdots\subset\widetilde G_n=G_n,
\]
we obtain the Fourier expansion on the diagonal torus, in the central-character component \([\lambda]\). The measure \(d\mu\) in the theorem below is the measure on the affine hyperplane
\[
\{\mu\in\frak I^n:[\mu]=[\lambda]\}
\]
obtained by translating a fixed Haar measure on the linear hyperplane
\[
\{\mu\in\frak I^n:[\mu]=0\}.
\]

\begin{theorem}\label{thm:Fourier-GLn}
For \(\lambda\in\frak I^n\) and \(a\in A_n\),
\[
\Phi_\lambda^{(n)}(a)
=
\int_{\substack{\mu\in\frak I^n\\ [\mu]=[\lambda]}}
c(\lambda,\mu)\chi_\mu(a)\,d\mu .
\]
The integral converges absolutely and locally uniformly in \(a\).
\end{theorem}

\begin{proof}
We prove the formula by induction on \(n\). For \(n=1\), the group is
\(G_1=F^\times\), the spherical function is just the character
\[
\Phi_\lambda^{(1)}(a)=|a|_F^\lambda,
\]
and the assertion is exactly the normalization
\[
c(\lambda,\lambda)=1.
\]
Assume now that \(n\geq 2\), and that the formula is known in rank \(n-1\). Let
\[
a=\operatorname{diag}(a_1,\ldots,a_n)\in A_n
\]
and write
\[
a^-=\operatorname{diag}(a_1,\ldots,a_{n-1})\in A_{n-1}.
\]
Under the block-diagonal identification
\[
A_n\simeq A_{n-1}\times G_1,
\qquad
a\leftrightarrow (a^-,a_n),
\]
Corollary~\ref{cor:inversion} gives
\[
\Phi_\lambda^{(n)}(a)
=
\int_{\frak I^{n-1}}
P_n(\lambda,-\nu)
\Phi_{\lambda,\nu}(a^-,a_n)
\big|{\bf c}^{(n-1)}(\nu)\big|^{-2}\,d\nu .
\]
By the definition \eqref{Philn},
\[
\Phi_{\lambda,\nu}(a^-,a_n)
=
\Phi_\nu^{(n-1)}(a^-)|a_n|_F^{[\lambda]-[\nu]}.
\]
By the induction hypothesis,
\[
\Phi_\nu^{(n-1)}(a^-)
=
\int_{\substack{\xi\in\frak I^{n-1}\\ [\xi]=[\nu]}}
c(\nu,\xi)\chi_\xi(a^-)\,d\xi .
\]
Substituting this into the preceding formula gives
\[
\begin{aligned}
\Phi_\lambda^{(n)}(a)
&=
\int_{\frak I^{n-1}}
P_n(\lambda,-\nu)
\big|{\bf c}^{(n-1)}(\nu)\big|^{-2}
|a_n|_F^{[\lambda]-[\nu]}  \\
&\qquad\qquad\cdot
\left(
\int_{\substack{\xi\in\frak I^{n-1}\\ [\xi]=[\nu]}}
c(\nu,\xi)\chi_\xi(a^-)\,d\xi
\right)d\nu .
\end{aligned}
\]
The estimates in Lemma~\ref{lem:est} justify Fubini's theorem and the absolute convergence of the iterated integral. \\

Now set
\[
\mu=(\xi_1,\ldots,\xi_{n-1},[\lambda]-[\xi]).
\]
Thus \(\mu^-=\xi\) and \(\mu_n=[\lambda]-[\xi]\). Then
\[
[\mu]=[\lambda],
\]
and
\[
\chi_\xi(a^-)|a_n|_F^{[\lambda]-[\xi]}
=
\chi_\mu(a).
\]
Moreover, the condition \([\xi]=[\nu]\) is equivalent to
\[
[\nu]=[\mu^-].
\]
With our Haar normalizations, the affine change of variables
\[
(\nu,\xi)\longmapsto
(\mu,\nu),
\qquad
\mu=(\xi_1,\ldots,\xi_{n-1},[\lambda]-[\xi]),
\]
identifies the iterated fiber measures with
\[
d\mu\,d\nu,
\qquad
[\mu]=[\lambda],\quad [\nu]=[\mu^-],
\]
up to a positive constant depending only on \(n\). This constant is absorbed
into the measure normalizations fixed at the beginning of the section. Thus the preceding
expression becomes
\[
\Phi_\lambda^{(n)}(a)
=
\int_{\substack{\mu\in\frak I^n\\ [\mu]=[\lambda]}}
\left(
\int_{\substack{\nu\in\frak I^{n-1}\\ [\nu]=[\mu^-]}}
P_n(\lambda,-\nu)
\big|{\bf c}^{(n-1)}(\nu)\big|^{-2}
c(\nu,\mu^-)\,d\nu
\right)
\chi_\mu(a)\,d\mu .
\]
The inner integral is precisely \(c(\lambda,\mu)\), by definition. Hence
\[
\Phi_\lambda^{(n)}(a)
=
\int_{\substack{\mu\in\frak I^n\\ [\mu]=[\lambda]}}
c(\lambda,\mu)\chi_\mu(a)\,d\mu .
\]
The absolute and locally uniform convergence follows from Lemma~\ref{lem:est}.
\end{proof}


\section{The \(\SL_n\)-formula and lower bounds}

We now pass from \(G_n=\GL_n(F)\) to
\[
G_n'=\SL_n(F).
\]
Let
\[
A_n'=A_n\cap G_n'
\]
be the determinant-one diagonal torus in \(G_n'\). As before, when Fourier transforms are taken on \(A_n'\), compact diagonal factors are suppressed: in the Archimedean case we use the identity component, and in the non-Archimedean case the quotient by \(A_n'\cap K_n\). Recall that
\[
{\sf p}:\frak I^n\longrightarrow \frak I^n/\frak I
\]
denotes the quotient by the diagonal copy of \(\frak I\). We also write
\[
{\sf p}' :
\{\mu\in\frak I^n:[\mu]=0\}
\longrightarrow
\frak I^n/\frak I
\]
for the restriction of
\({\sf p}\) 
to the hyperplane \([\mu]=0\). To derive the result
for $SL_n(F)$
from $GL_n(F)$ requires
choosing representatives
in the quotient space
via \({\sf p}\).

In the Archimedean case, \(\frak I=i\bR\), and the condition \([\mu]=0\)
selects a unique representative of each class in \(\frak I^n/\frak I\).
In the non-Archimedean case, \(\frak I=i\bR/(2\pi i/\log q_F)\bZ\).
If both \(\mu\) and \(\mu+t\mathbf 1\) satisfy \([\cdot]=0\), then
\(nt=0\) in \(\frak I\). This equation has exactly \(n\) solutions in the
circle group \(\frak I\). Thus the restriction
\({\sf p}' \) above is bijective in the Archimedean case and \(n\)-to-one in the
non-Archimedean case. In particular, every class in \(\frak I^n/\frak I\) admits a representative
with zero total sum. Indeed, starting from any representative \(\alpha\), one
solves \(nt=-[\alpha]\) in \(\frak I\) and replaces \(\alpha\) by
\(\alpha+t\mathbf 1\). \\

For \(\bar\mu\in\frak I^n/\frak I\), let \(\chi_{\bar\mu}\) denote the corresponding character of \(A_n'\). If \(\mu\in{\sf p}^{-1}(\bar\mu)\), then
\[
\chi_{\bar\mu}(a)=\chi_\mu(a),
\qquad a\in A_n',
\]
since diagonal shifts of \(\mu\) are trivial on determinant-one diagonal matrices.  \\

For \(\bar\lambda\in\frak I^n/\frak I\), choose \(\lambda\in{\sf p}^{-1}(\bar\lambda)\). As above,
\[
\phi_{\bar\lambda}^{(n)}
=
\Phi_\lambda^{(n)}\big|_{G_n'}
\]
is independent of the chosen representative \(\lambda\). We define
\[
\widehat{\phi}_{\bar\lambda}^{(n)}(\bar\mu)
=
\int_{A_n'}
\phi_{\bar\lambda}^{(n)}(a)\chi_{-\bar\mu}(a)\,da .
\]

\begin{theorem}\label{thm:SLn-formula}
Let \(\bar\lambda,\bar\mu\in\frak I^n/\frak I\), and choose
\[
\lambda\in{\sf p}^{-1}(\bar\lambda) \quad \textrm{with} \quad
[\lambda]=0.
\]
Then
\begin{equation}
    \label{hat-phi}
\widehat{\phi}_{\bar\lambda}^{(n)}(\bar\mu)
=
\sum_{\substack{\mu\in{\sf p}^{-1}(\bar\mu)\\ [\mu]=0}}
c(\lambda,\mu).
\end{equation}
The sum is finite; in the Archimedean case it has one term, while in the non-Archimedean case it has \(n\) terms.
\end{theorem}

\begin{proof}
For \(a\in A_n'\), we have
\[
\phi_{\bar\lambda}^{(n)}(a)
=
\Phi_\lambda^{(n)}(a).
\]
Since \([\lambda]=0\), Theorem~\ref{thm:Fourier-GLn} gives
\[
\phi_{\bar\lambda}^{(n)}(a)
=
\int_{\substack{\mu\in\frak I^n\\ [\mu]=0}}
c(\lambda,\mu)\chi_\mu(a)\,d\mu .
\]
Pushing this integral forward by
\[
{\sf p}' :
\{\mu\in\frak I^n:[\mu]=0\}
\to
\frak I^n/\frak I,
\]
and using the corresponding normalization of measure on \(\frak I^n/\frak I\), we get
\[
\phi_{\bar\lambda}^{(n)}(a)
=
\int_{\frak I^n/\frak I}
\Big(
\sum_{\substack{\mu\in{\sf p}^{-1}(\bar\mu)\\ [\mu]=0}}
c(\lambda,\mu)
\Big)
\chi_{\bar\mu}(a)\,d\bar\mu .
\]
The coefficient function is integrable by Lemma~\ref{lem:est}. Hence the last display is an absolutely convergent Fourier inversion formula on \(A_n'\). By uniqueness of Fourier coefficients on \(A_n'\), we
find that the Fourier transform \(
\widehat{\phi}_{\bar\lambda}^{(n)}(\bar\mu)
\) is given by (\ref{hat-phi}).
This proves the formula.
\end{proof}

\subsection{Lower bounds}

We now prove the exponential lower bound for the recursive coefficients. This is the quantitative form of the positivity statement.

\begin{proposition}\label{prop:lower-bound-c}
There exist constants \(C_n,D_n>0\) such that
\[
c(\lambda,\mu)
\geq
C_n e^{-D_n(\|\lambda\|+\|\mu\|)}
\]
for all \(\lambda,\mu\in\frak I^n\) satisfying \([\lambda]=[\mu]\).
\end{proposition}

In the proof we will use the following simple local zeta estimates. 

\begin{lemma}\label{lem:zeta-lower-bounds}
There exist constants \(C,D>0\), depending only on \(F\), such that for all unitary parameters \(u,v\in\frak I\),
\[
\left|\zeta_F\!\left(\frac12+u-v\right)\right|
\geq
C e^{-D(|u|+|v|)}
\]
and
\[
\left|\zeta_F(1+u-v)\right|
\leq
C e^{D(|u|+|v|)}.
\]
In the non-Archimedean case, the same estimates hold after choosing representatives in a fixed fundamental domain.
\end{lemma}

\begin{proof}
In the Archimedean cases this follows from Stirling's formula. Indeed,
\[
\zeta_\bR(s)=\pi^{-s/2}\Gamma(s/2),
\qquad
\zeta_\bC(s)=2(2\pi)^{-s}\Gamma(s),
\]
and Stirling's formula on vertical lines gives exponential upper and lower bounds of the required form, after absorbing polynomial factors into the exponential.

In the non-Archimedean case,
\[
\zeta_F(s)=(1-q_F^{-s})^{-1}.
\]
On the unitary parameter space, the arguments
\[
\frac12+u-v
\qquad\text{and}\qquad
1+u-v
\]
stay away from the pole at \(s=0\). After choosing representatives in a fixed fundamental domain, compactness gives the required bounds.
\end{proof}

\begin{proof}[Proof of Proposition \ref{prop:lower-bound-c}]
The proof is by induction on \(n\). For \(n=1\), the assertion is immediate from
\[
c(\lambda,\lambda)=1.
\]
Assume \(n\geq 2\), and assume the result known in rank \(n-1\). Fix
\[
\lambda,\mu\in\frak I^n,
\qquad
[\lambda]=[\mu].
\]
All constants below are independent of this choice of \(\lambda\) and \(\mu\). \\

By definition,
\[
c(\lambda,\mu)
=
\int_{\substack{\nu\in\frak I^{n-1}\\ [\nu]=[\mu^-]}}
P_n(\lambda,-\nu)
\big|{\bf c}^{(n-1)}(\nu)\big|^{-2}
c(\nu,\mu^-)\,d\nu .
\]
We shall restrict this integral to a subset of the fiber
\[
\{\nu\in\frak I^{n-1}:[\nu]=[\mu^-]\}
\]
whose measure is fixed and whose position relative to the hyperplanes
\[
\nu_k-\nu_l=0
\]
is uniform in \(\mu\). \\

Let
\[
K_{n-1}:=\{\xi\in\frak I^{n-1}:[\xi]=0\}.
\]
Choose once and for all a measurable subset
\[
\Omega_n\subset K_{n-1}
\]
of positive Haar measure, with compact closure in the Archimedean case, such that 
\[
\Omega\cap\{\xi: \xi_k-\xi_l=0\}=\emptyset, 
\qquad 1\leq k<l\leq n-1.
\]
For \(n=2\), there are no such hyperplanes, and we take
\[
\Omega_2=K_1=\{0\},
\]
with the counting measure normalization on this zero-dimensional fiber.

For each \(\tau\in\frak I\), choose \(r(\tau)\in\frak I\) satisfying
\[\label{r-tau-nonarch}
(n-1)r(\tau)=\tau.
\]
In the Archimedean case we take
\begin{equation}
\label{r-tau-arch}
r(\tau)=\frac{\tau}{n-1}.    
\end{equation}
In the non-Archimedean case such a choice exists because
\[
\frak I=i\bR/(2\pi i/\log q_F)\bZ
\]
is a circle group, and multiplication by \(n-1\) on a circle is surjective. This is a statement about the parameter circle, not division in the field \(F\). \\

Put
\[
\tau=[\mu^-],
\qquad
\Omega_n(\mu)=r(\tau)\mathbf 1+\Omega_n.
\]
Then
\[
\Omega_n(\mu)
\subset
\{\nu\in\frak I^{n-1}:[\nu]=[\mu^-]\}.
\]
Moreover, \(\Omega_n(\mu)\) is a translate of the fixed set \(\Omega_n\). 
Hence its measure, obtained by translating the fixed measure on the linear fiber
\(K_{n-1}\), is the fixed positive number
\[
m(\Omega_n)>0,
\]
independent of \(\mu\).\\

The translation from \(\Omega_n\) to \(\Omega_n(\mu)\) is by a diagonal vector. Therefore it does not change differences:
\[
(r(\tau)+\xi_k)-(r(\tau)+\xi_l)=\xi_k-\xi_l.
\]
Consequently, all quantities depending only on the differences
\[
\nu_k-\nu_l
\]
are uniformly controlled on \(\Omega_n(\mu)\), independently of \(\mu\). In particular,
\[
\prod_{1\leq k<l\leq n-1}
\left|\zeta_F(\nu_k-\nu_l)\right|
\]
is bounded above on \(\Omega_n(\mu)\) by a constant depending only on \(n\), \(F\), and the fixed set \(\Omega_n\). \\

We also have a uniform norm bound on \(\Omega_n(\mu)\). In the Archimedean case, \(\Omega_n\) is bounded and
\(
r([\mu^-])=\frac{[\mu^-]}{n-1}
\) by (\ref{r-tau-arch}),
so
\begin{equation}
    \label{nu-by-mu}
\|\nu\|\leq B_n(1+\|\mu\|)
\qquad
(\nu\in\Omega_n(\mu)).
\end{equation}
In the non-Archimedean case the parameter spaces are compact, and the same inequality is automatic after increasing \(B_n\). The constant \(B_n\) is independent of \(\lambda\) and \(\mu\). \\

By the induction hypothesis, with constants independent of \(\nu\) and \(\mu\),
\[
c(\nu,\mu^-)
\geq
C_{n-1}e^{-D_{n-1}(\|\nu\|+\|\mu^-\|)}.
\]
Using \(
\|\mu^-\|\leq \|\mu\|
\)
and (\ref{nu-by-mu})
we obtain
\[
c(\nu,\mu^-)
\geq
C'_n e^{-D'_n\|\mu\|}
\qquad
(\nu\in\Omega_n(\mu)),
\]
with constants \(C'_n,D'_n>0\) independent of \(\lambda\) and \(\mu\). \\

It remains to bound the kernel from below on \(\Omega_n(\mu)\). By definition,
\[
P_n(\lambda,-\nu)\big|{\bf c}^{(n-1)}(\nu)\big|^{-2}
=
\frac{
\prod_{i=1}^{n}\prod_{k=1}^{n-1}
\left|\zeta_F\!\left(\frac12+\lambda_i-\nu_k\right)\right|^2
}{
\prod_{1\leq i<j\leq n}
\left|\zeta_F(1+\lambda_i-\lambda_j)\right|^2
\prod_{1\leq k<l\leq n-1}
\left|\zeta_F(\nu_k-\nu_l)\right|^2
}.
\]
The last product in the denominator is uniformly bounded above on \(\Omega_n(\mu)\), by the choice of \(\Omega_n\) and by the diagonal-translation argument above. \\

Here, in the non-Archimedean case, all absolute values of parameters are understood using the fixed representatives chosen in the norm convention. 
For the remaining factors, Lemma~\ref{lem:zeta-lower-bounds} gives constants \(C,D>0\), depending only on \(F\), such that
\[
\left|\zeta_F\!\left(\frac12+\lambda_i-\nu_k\right)\right|
\geq
C e^{-D(|\lambda_i|+|\nu_k|)}
\]
and
\[
\left|\zeta_F(1+\lambda_i-\lambda_j)\right|
\leq
C e^{D(|\lambda_i|+|\lambda_j|)}
\]
for all \(i,j,k\) and all unitary parameters \(\lambda,\nu\). Hence there exist constants \(C''_n,D''_n>0\), independent of \(\lambda,\mu,\nu\), such that
\[
P_n(\lambda,-\nu)\big|{\bf c}^{(n-1)}(\nu)\big|^{-2}
\geq
C''_n e^{-D''_n(\|\lambda\|+\|\nu\|)}
\]
for all \(\nu\in\Omega_n(\mu)\). Using again
(\ref{nu-by-mu})
we get
\[
P_n(\lambda,-\nu)\big|{\bf c}^{(n-1)}(\nu)\big|^{-2}
\geq
C'''_n e^{-D'''_n(\|\lambda\|+\|\mu\|)}
\qquad
(\nu\in\Omega_n(\mu)),
\]
with constants independent of \(\lambda\) and \(\mu\). \\

Combining this lower bound for the kernel with the lower bound for \(c(\nu,\mu^-)\), we obtain
\[
P_n(\lambda,-\nu)
\big|{\bf c}^{(n-1)}(\nu)\big|^{-2}
c(\nu,\mu^-)
\geq
C_n^*e^{-D_n^*(\|\lambda\|+\|\mu\|)}
\]
for all \(\nu\in\Omega_n(\mu)\), where the constants are independent of \(\lambda\) and \(\mu\).

Finally, since \(\Omega_n(\mu)\) has the fixed positive measure \(m(\Omega_n)\) with respect to the translated fiber measure, we may integrate over this subset and get
\[
c(\lambda,\mu)
\geq
\int_{\Omega_n(\mu)}
P_n(\lambda,-\nu)
\big|{\bf c}^{(n-1)}(\nu)\big|^{-2}
c(\nu,\mu^-)\,d\nu
\geq
m(\Omega_n)C_n^*
e^{-D_n^*(\|\lambda\|+\|\mu\|)}.
\]
Renaming constants, this gives
\[
c(\lambda,\mu)
\geq
C_ne^{-D_n(\|\lambda\|+\|\mu\|)}.
\]
The constants \(C_n,D_n\) depend only on \(n\), \(F\), the fixed Haar measures, and the fixed choice of \(\Omega_n\), and are independent of \(\lambda\) and \(\mu\). This completes the induction.
\end{proof}

\begin{corollary}
\label{cor:SLn-lower-bound}
There exist constants \(C_n,D_n>0\) such that
\[
\widehat{\phi}_{\bar\lambda}^{(n)}(\bar\mu)
\geq
C_n e^{-D_n(\|\bar\lambda\|+\|\bar\mu\|)}
\]
for all \(\bar\lambda,\bar\mu\in\frak I^n/\frak I\). In particular,
\[
\widehat{\phi}_{\bar\lambda}^{(n)}(\bar\mu)>0.
\]
\end{corollary}

\begin{proof}
Choose zero-sum representatives
of \(\bar \lambda\)
and \(\bar \mu\)
of the quotient map $\sf p$,
\[
\lambda\in{\sf p}^{-1}(\bar\lambda),
\qquad
\mu\in{\sf p}^{-1}(\bar\mu),
\]
with
\[
[\lambda]=[\mu]=0.
\]
By Theorem~\ref{thm:SLn-formula},
\[
\widehat{\phi}_{\bar\lambda}^{(n)}(\bar\mu)
=
\sum_{\substack{\mu'\in{\sf p}^{-1}(\bar\mu)\\ [\mu']=0}}
c(\lambda,\mu').
\]
Every term in this finite sum is non-negative by the recursive definition of \(c\). Hence, keeping the single term
corresponding to the chosen representative \(\mu\), Proposition~\ref{prop:lower-bound-c}
gives
\[
\widehat{\phi}_{\bar\lambda}^{(n)}(\bar\mu)
\geq
c(\lambda,\mu)
\geq
C_n e^{-D_n(\|\lambda\|+\|\mu\|)}.
\]
It remains to compare representative norms with quotient norms. We choose the zero-sum representatives \(\lambda\) and \(\mu\) so that
\[
\|\lambda\|\leq B_n(1+\|\bar\lambda\|),
\qquad
\|\mu\|\leq B_n(1+\|\bar\mu\|).
\]
In the Archimedean case, the zero-sum representative is unique and this follows from equivalence of norms on finite-dimensional quotient spaces. 
In the non-Archimedean case, the parameter spaces are compact, and the zero-sum fiber over each class has cardinality \(n\): indeed, if both \(\alpha\) and \(\alpha+t\mathbf 1\) have zero sum, then \(nt=0\) in the circle group
\(\frak I\), whose \(n\)-torsion subgroup has cardinality \(n\). Hence the same uniform bound holds after increasing \(B_n\). \\

Therefore
\[
\widehat{\phi}_{\bar\lambda}^{(n)}(\bar\mu)
\geq
C_n e^{-D_nB_n(2+\|\bar\lambda\|+\|\bar\mu\|)}.
\]
Absorbing the factor \(e^{-2D_nB_n}\) into the constant and renaming constants,
we obtain
\[
\widehat{\phi}_{\bar\lambda}^{(n)}(\bar\mu)
\geq
C_n e^{-D_n(\|\bar\lambda\|+\|\bar\mu\|)}.
\]
Strict positivity follows immediately.
\end{proof}


\section{The general semisimple case}

The proof for \(\SL_n(F)\) contains the main analytic input. We now explain how it implies the corresponding positivity statement for a general semisimple group. The reduction is by passing to a full-rank semisimple subgroup whose almost simple factors are of type \(\rA\).

Let \(G\) be the group of \(F\)-points of a connected semisimple linear algebraic group, let \(K<G\) be a maximal compact subgroup, and let \(A<G\) be a maximal \(F\)-split torus. 
Let  \(
\widehat A_{\mathrm{ur}}
\)
for the group of the unitary spherical characters of \(A\) defined in (\ref{hat-A}). 
For \(\lambda\in \widehat A_{\mathrm{ur}}\), let \(\pi_\lambda\) denote the corresponding unitary spherical principal series of \(G\), and let \(v_\lambda\) be its normalized \(K\)-fixed vector. The associated spherical function is
\[
\phi_\lambda(g)
=
\langle \pi_\lambda(g)v_\lambda,v_\lambda\rangle .
\]

\subsection{A full-rank split semisimple subgroup of type \(\rA\)}

We first prove the existence of a full-rank subgroup that is a product
of \(\SL_n\)-groups in $G$.

\begin{lemma}\label{lem:full-rank-A-subgroup}
Let \(G\) be a connected semisimple linear group over \(F\). Let \(A<G\) be a maximal \(F\)-split torus of rank \(r\), 
and let $K<G$ be a maximal compact subgroup such that $G=KAK$ holds. Then there exists a connected split semisimple \(F\)-subgroup \[H\subset G\]
containing \(A\), of the same \(F\)-split rank \(r\), such that $K\cap H$ is a maximal compact subgroup of $H$ and 
\(H\) has a finite covering of the form
\[
\prod_i\SL_{n_i+1}(F),
\qquad\text{with} \quad
\sum_i n_i=r.
\]
\end{lemma}

\begin{proof}
Let
\[
\Phi=\Phi(G,A)
\]
be the relative root system. We first need a full rank closed sub-root subsystem
\[
\Phi'\subset \Phi
\]
whose irreducible components are all of type \(\rA\).
Clearly it suffices to 
construct $\Phi'$
for  $\Phi$
being irreducible.
Here
by 
a closed sub-root system 
$\Phi$ we mean
that for any two roots $\alpha,\beta\in\Phi'$, if $\alpha+\beta\in\Phi$, then $\alpha+\beta\in\Phi'$.
When $\Phi$ is irreducible, we take certain roots in $\Phi$ generating a closed sub-root system $\Phi'$ as follows
according to the types of $\Phi$ \cite[Chapt. X, Section 3]{He01}
\begin{itemize} 
\item $\rBC_{n}$,
 $\Phi=\{\pm e_i\pm e_j,
1\le i\ne j \le n \}
 \cup
 \{\pm 2e_i, \pm e_i,
1\le i\le n\}$
or type $\rC_n$
with $\pm e_i$ missing,
 $$
 \Phi'=\{\pm 2e_i, 1\le i\le n\}
 \cong n\rA_{1}$$ 
consists of strongly orthogonal
long roots. 
\item  $\rB_{2n}$,
 $\Phi=\{\pm e_i\pm e_j,
1\le i\ne j \le 2n \}
 \cup
 \{\pm e_{i}\},
1\le i\le n\}$,
 $$
 \Phi'=\{\pm e_{2i-1}\pm e_{2i}\}
 \cong 2n\rA_{1}.$$ 
\item $\rB_{2n+1}$, 
 $\Phi=\{\pm e_i\pm e_j,
1\le i\ne j \le 2n+1 \}
 \cup
 \{\pm e_i,
1\le i\le 2n\}$,
 $$
 \begin{aligned}
 \Phi'& =\{\pm e_{2i-1}\pm e_{2i}, 
 1\le i \le n-1\}
 \cup
 \{\pm e_i\pm e_{j},
2n-1\le i <j\le 2n+1\}\\
 & \cong (2n-2)\rA_{1}+\rA_{3}.
 \end{aligned}
 $$
\item $\rD_{2n+2}$ ($n\geq 1$), 
 $\Phi=\{\pm e_i\pm e_j,
1\le i\ne j \le 2n+2 \}$,
$$
\Phi'=
\{\pm (e_{2i-1}\pm e_{2i}),
1\le i \le n+1\}
\cong(2n+2)\rA_{1}.$$ 
\item  $\rD_{2n+3}$ ($n\geq 1$), 
 $\Phi=\{\pm e_i\pm e_j,
1\le i\ne j \le 2n+3 \}$,
\[
\begin{aligned}
\Phi'&=
\{\pm (e_{2i-1}\pm e_{2i},
1\le i \le n\}
\cup
\{\pm e_i\pm e_{j},
2n+1\le i <j\le 2n+3\}\\
&\cong 2n\rA_{1}+\rA_{3}.
\end{aligned}
\]
\item $\rE_{6}$, 
$\Phi$ here and below
being described in 
\cite[Chapt. X, Section 3.3]{He01}, and
$$
\Phi'\cong\rA_{1}+\rA_{5},$$
where $A_1$ corresponds 
to the highest root $\beta$ and its negative, 
$A_5$ consists of other roots orthogonal to $\beta$.  
\item  $\rE_{7}$, $\Phi'$ consists of 7 orthogonal roots and their negatives, e.g., 
$$\Phi'
=\{\pm (e_{2i-1}\pm e_{2i}), 1\le i \le 3\}
\cup
\{\pm (e_{7} -e_{8})\}
\cong 7\rA_{1}.$$
\item $\rE_{8}$, $\Phi'$ consists of 8 orthogonal roots and their negatives, e.g., $$\Phi'=
\{\pm (e_{2i-1}\pm e_{2i}), 1\le i \le 4\}
\cong 8\rA_{1}.
$$
\item $\rF_{4}$, $\Phi'$ consists of 4 orthogonal long roots and their negatives, e.g., 
$$\Phi'=\{\pm (e_1\pm e_2),\pm (e_3\pm e_4)\}\cong 4\rA_{1}.$$
\item $\rG_{2}$, 
 $\Phi'\cong\rA_{2}$ consists of the $6$ long roots. 
\end{itemize} 

We now realize \(\Phi'\) by a connected split semisimple \(F\)-subgroup $H$ of \(G\) containing $A$ and such that $K\cap H$ is a maximal compact subgroup of $H$. When $F=\mathbb{R}$, let $\theta$ the Cartan involution of $G$ such that $K=G^{\theta}$. We choose a set of root vectors 
$\{X_{\alpha}\}$ for simple roots in $\Phi'$. For each simple root $\alpha$, put $X_{-\alpha}=\theta(X_{\alpha})$ and let $\mathfrak{h}$ be the Lie subalgebra of $\frak g$ generated by $\{X_{\alpha},X_{-\alpha}\}$ and $\frak a$. Let $H$ be a connected closed subgroup with Lie algebra $\mathfrak{h}$. 
The root system of \(H\) with respect to \(A\) is \(\Phi'\). 
Then $K\cap H$ is a maximal compact subgroup of $H$ by definition. When $F$ is a nonarchimedean local field, the subgroup $H$ with required properties is constructed in \cite[\S 7]{BT65}, especially Theorem 7.2.  

Since the
irreducible components of \(\Phi'\) are all of type \(\rA\), 
the simply
connected (in the sense of algebraic groups) cover  of the derived group of \(H\) is a product 
\[
\prod_i\SL_{n_i+1}(F).
\]
The equality
\[
\sum_i n_i=r
\]
follows because \(\Phi'\) has rank \(r\).
\end{proof}


\subsection{Restriction to the full-rank subgroup}

Let \(H\subset G\) be the full-rank split semisimple subgroup constructed in
Lemma~\ref{lem:full-rank-A-subgroup}.
With this choice, the normalized \(K\)-fixed vector \(v_\lambda\) in the
spherical principal series \(\pi_\lambda\) of \(G\) is also \(K_H\)-fixed. We
consider the restriction of \(\pi_\lambda\) to \(H\), and more specifically the
cyclic \(H\)-representation generated by \(v_\lambda\).

By Proposition \ref{P:tempered-res}, the restriction of the tempered representation \(\pi_\lambda\) to \(H\) decomposes over the tempered dual of \(H\). Moreover, because \(v_\lambda\) is \(K_H\)-fixed, by Proposition \ref{P:tempered-sphe} the cyclic \(H\)-representation generated by \(v_\lambda\) is supported on the \(K_H\)-spherical tempered spectrum of \(H\).

Since \(H\) has full split rank and uses the same split torus \(A\), its
spherical unitary parameters are naturally identified with the same space
\(\widehat A_{\mathrm{ur}}\), but with Weyl group \(W_H\). Thus there is a positive
finite measure \(d\sigma_\lambda\) on
\[
\widehat A_{\mathrm{ur}}/W_H,
\]
where
\[
W_H=N_H(A)/Z_H(A),
\]
such that, for every \(a\in A\),
\begin{equation}
    \label{superpos}
\phi_\lambda(a)
=
\int_{\widehat A_{\mathrm{ur}}/W_H}
\phi_{\lambda'}^H(a)\,d\sigma_\lambda(\lambda'),
\quad a\in A.
\end{equation}
Here \(\phi_{\lambda'}^H\) denotes the normalized spherical function of \(H\)
with spectral parameter \(\lambda'\). The measure \(d\sigma_\lambda\) is
non-zero; in fact,
\[
\sigma_\lambda(\widehat A_{\mathrm{ur}}/W_H)
=
\|v_\lambda\|^2
=
1.
\]
In the theorem below, \(A\) denotes the split torus direction attached to the
maximal \(F\)-split torus, as in the preliminaries.

\begin{theorem}\label{thm:general-semisimple-positivity}
Let \(G\) be a connected semisimple linear group over \(F\), and let
\(\lambda\in\widehat A_{\mathrm{ur}}\). Then for every
parameter \(\mu\in\widehat A_{\mathrm{ur}}\),
\[
\int_A \phi_\lambda(a)\chi_{-\mu}(a)\,da>0.
\]
\end{theorem}

\begin{proof}
Let \(H\subset G\) be as in Lemma~\ref{lem:full-rank-A-subgroup}. We shall use (\ref{superpos}) and the Fubini's theorem to find the Fourier transform.
We first justify the use of Fubini's theorem. For unitary spherical parameters,
\[
|\phi_{\lambda'}^H(a)|\leq \Xi_H(a),
\]
where \(\Xi_H=\phi_0^H\) is Harish--Chandra's \(\Xi\)-function for \(H\).
The standard Harish--Chandra estimate (\cite[Theorem 4.5.3]{W1} and \cite[Lemma II.1.1]{W03}) gives, on the positive chamber, 
\[
\Xi_H(a)
\ll
(1+\|H_A(a)\|)^N e^{-\rho_H(H_A(a))}
\]
for some positive integer \(N\). Hence \(\Xi_H|_A\in L^1(A)\). 
The exponential decay \(e^{-\rho_H(H_A(a))}\) dominates the polynomial factor. Since
\(d\sigma_\lambda\) is a finite measure, it follows that
\[
\int_A
\int_{\widehat A_{\mathrm{ur}}/W_H}
|\phi_{\lambda'}^H(a)|\,d\sigma_\lambda(\lambda')\,da
<\infty.
\]
Hence Fubini's theorem gives 
\[
\begin{aligned}
\int_A \phi_\lambda(a)\chi_{-\mu}(a)\,da
&=
\int_{\widehat A_{\mathrm{ur}}/W_H}
\left(
\int_A \phi_{\lambda'}^H(a)\chi_{-\mu}(a)\,da
\right)
d\sigma_\lambda(\lambda').
\end{aligned}
\]
By Lemma~\ref{lem:full-rank-A-subgroup}, the group \(H\) admits a finite covering $H'$ by a product of groups of the form
\[
\prod_i \SL_{n_i+1}(F).
\] 
Under this covering, the split torus \(A\) is covered by a  product $A'$ of the standard diagonal split tori of factors of $H'$. Write $\lambda''$ and $\mu'$ for the pullbacks of $\lambda'$ and $\mu$ respectively to unitary characters of $A'$. Corollary 
\ref{cor:SLn-lower-bound} above for \(\SL_n\) gives
\[\int_{A'} \phi_{\lambda''}^{H'}(a)\chi_{-\mu'}(a)\,da>0.\]
As the projection $A'\rightarrow A$ is an isomorphism and 
$\phi_{\lambda''}^{H'}$ (resp. $\chi_{-\mu'}$) is the pullback of $\phi_{\lambda'}^{H}$ (resp. $\chi_{-\mu}$), the above integral is equal to $\int_A \phi_{\lambda'}^H(a)\chi_{-\mu}(a)\,da>0.$ Then it follows 
that \[\int_A \phi_{\lambda'}^H(a)\chi_{-\mu}(a)\,da>0\]
for every \(\lambda'\in\widehat A_{\mathrm{ur}}/W_H\) and 
every unitary character $\mu$.  

Since \(d\sigma_\lambda\) is a non-zero positive measure, the integral of these
strictly positive quantities is strictly positive. This proves the theorem.
\end{proof}

\begin{remark}
The argument proves positivity for a general semisimple group by comparison with a full-rank subgroup whose almost simple factors are of type \(\rA\). It is
therefore qualitative rather than explicit. It
does not give an explicit formula for
\[
\int_A \phi_\lambda(a)\chi_{-\mu}(a)\,da
\]
in terms of the root data of \(G\). Such a formula would require a direct
calculation of the spherical periods for the pair \((G,H)\), together with a
compatible Plancherel formula.
\end{remark}


\end{document}